\documentclass[oneside,english,reqno]{amsart}
\usepackage[T1]{fontenc}
\usepackage[latin9]{inputenc}
\usepackage{amsthm}
\usepackage[english]{babel}
\usepackage{amsfonts}
\usepackage{amssymb}
\usepackage{amsthm}
\usepackage{newlfont}

\usepackage[colorlinks=true,urlcolor=blue,citecolor=blue,linkcolor=black,linktocpage,pdfpagelabels,bookmarksnumbered,bookmarksopen]{hyperref}
 \usepackage{bigints}
\usepackage[normalem]{ulem}

\newcommand{\Wp}{W_0^{1,p}(\Omega)}

\newcommand{\Om}{\Omega}
\newcommand{\tOm}{\tilde{\Omega}}

\newcommand{\norm}[1]{\left\Vert#1\right\Vert}

\newcommand{\Zueps}{Z_{u}^{\bar{\varepsilon}}}

\numberwithin{equation}{section} 
\numberwithin{figure}{section} 
\newtheorem{thm}{Theorem}[section]
\newtheorem{cor}[thm]{Corollary}
\newtheorem{lem}[thm]{Lemma}
\newtheorem{prop}[thm]{Proposition}
\theoremstyle{definition}
\newtheorem{defn}[thm]{Definition}
\theoremstyle{remark}
\newtheorem{rem}[thm]{Remark}
\theoremstyle{example}

\theoremstyle{cexample}

\theoremstyle{procedure}

\numberwithin{equation}{section}
\newcommand{\Real}{\mathbb{R}}
\makeatother

\begin{document}

\title[Existence and regularity for quasilinear elliptic problems]{Existence and regularity for a general class of quasilinear elliptic problems involving the hardy potential}
\author{giusy chirillo, luigi montoro, luigi muglia, berardino sciunzi}
\address{giusy chirillo, luigi montoro, luigi muglia, berardino sciunzi: Dipartimento di Matematica, Universit\'a della Calabria, 87036 Arcavacata di Rende (CS), ITALY}
\email{chirillo@mat.unical.it; montoro@mat.unical.it, muglia@mat.unical.it, sciunzi@mat.unical.it}
\begin{abstract}
In a very general quasilinear setting, we show that the regularizing
effect of a first order term causes the existence of energy solutions for
problems involving the Hardy potential and $L^1$ data. In the same setting
we study sharp (local and global) integral estimates for the second
derivatives of the solutions.
\end{abstract}
\thanks{Corresponding author: Berardino Sciunzi, Universit\'a della Calabria, via P. Bucci, 87036, Arcavacata di Rende (CS), Italy}

\maketitle

\section{Introduction}

The aim of this paper is to study the existence and regularity of positive weak solutions to the nonlinear quasilinear elliptic equation
\begin{equation}\label{P}
\begin{cases} \displaystyle
-\mbox{div}\left(A(|\nabla u|)\nabla u\right)+B(|\nabla u|) = \vartheta \frac{u^q}{|x|^p}+f &\mbox{ in } \Omega,\\
u \geq 0 &\mbox{ in } \Omega,\\
u = 0 &\mbox{ on } \partial \Omega,
\end{cases}
\end{equation}
where $\Omega$ is a smooth bounded domain in $\Real^N$, $0\in \Omega$, $\vartheta >0$, $p\in(1,N)$, $q\in (p-1,p)$, $f \geq 0$ and $f\in L^1(\Omega)$.

\noindent The real function $A: \Real^+ \rightarrow \Real^+$ is assumed to be of class $C^1(\Real^+)$ and fulfills the following assumptions
\begin{align}
\bullet & \ \ t \mapsto tA(t) \mbox{ is an increasing function} \label{crescenzaA};\\
\bullet & \ \exists K\geq 1 :\forall \eta \in \mathbb{R}^N \mbox{ with } |\eta| \geq K, \ \left|A( |\eta|)\eta\right| \leq c_1 |\eta|^{p-1}; \ \ \label{D10}\\
\bullet & \  \forall  \eta, \eta' \in \mathbb{R}^N, \left[A( |\eta|)\eta-A( |\eta'|)\eta'\right]\cdot [\eta-\eta'] \geq c_2  (|\eta|+|\eta'|)^{p-2}|\eta-\eta'|^2  \label{D2}.
\end{align}

\noindent We set $\overline{\Real}^+:= \Real^+ \cup \{0\}$. The real function $B: \overline{\Real}^+ \rightarrow \overline{\Real}^+$ is of class $C^1(\overline{\Real}^+)$ and it satisfies 
\begin{align}
\bullet & \  B(0)=0 \label{B0};\\
\bullet & \  B(t) \geq \sigma \ t^p \mbox{ for some } \sigma >0 \label{stimaBbasso};\\
\bullet & \  B'(t) \leq \hat{C}tA(t) \mbox{ for some } \hat{C}>0, \forall t \in \Real^+. \label{crescitaB'}
\end{align}

Let us start making explicit what we means by positive weak solution and stating our main result:

\begin{defn}
	We say that $u$ is a weak solution to 
	$$ -\mbox{div}\left(A(|\nabla u|)\nabla u\right)+B(|\nabla(u)|) = \vartheta \frac{u^q}{|x|^p}+f \quad \mbox{ in } \Omega,$$
	if $u\in W_0^{1,p}(\Omega)$ and 
	\begin{equation*}
	\begin{split}
	\int_{\Omega} A(|\nabla u|)(\nabla u, \nabla \varphi )\,dx +\int_{\Om} B(|\nabla u|)\varphi	\,dx  =   \vartheta &\int_{\Om} \frac{u^q}{|x|^p}\varphi \,dx + \int_{\Om} f \varphi \,dx \\ &\forall
	\varphi \in W_0^{1,p}(\Om) \cap L^{\infty}(\Omega). 
	\end{split}
	\end{equation*}
\end{defn}

\begin{thm}\label{ExistenceM}
	Let $f \in L^1(\Omega)$ a positive function; then for every $\vartheta >0$ there exists a weak solution $u \in \Wp $ to \eqref{P}.
\end{thm}
This result is proved in the Section 3.\\
Let us observe that one immediately recognized that the choices $A(t)=t^{p-2}$ and $B(t)=t^p$ leads to the supercritical problem
\begin{equation}
\begin{cases} \displaystyle
-\mbox{div}\left(|\nabla u|^{p-2}\nabla u\right)+|\nabla u|^p = \vartheta \frac{u^q}{|x|^p}+f &\mbox{ in } \Omega,\\
u \geq 0 &\mbox{ in } \Omega,\\
u = 0 &\mbox{ on } \partial \Omega,
\end{cases}
\end{equation}
for which existence and qualitative properties are proved in \cite{MontSciuPM}. Moreover taking into account \cite{Benetal, Dall'Aglio, DalMaso} we remark that our positive weak solutions are \emph{solutions obtained as limits of approximations} (SOLA) and therefore (in our case) equivalent to the entropy solutions.

\noindent The main elements of the proof are summarized in what follows.
\begin{itemize}
 \item In order to prove the existence of a solution $u_k$ for the truncated problem
 \begin{equation*}\displaystyle
	-\mbox{div} (A(|\nabla u_k|)\nabla u_k)+B(|\nabla u_k|) = \vartheta T_k\left(\frac{u_k^q}{|x|^p}\right)+T_k(f) \quad \mbox{ in } \Omega.
	\end{equation*} 
 we construct a sequence $(w_n)_{n\in\mathbb{N}}$ of solutions to the problem 
\begin{equation*}
-\mbox{div}\left(A(|\nabla w_n|)\nabla w_n\right)+\frac{B(|\nabla w_n|)}{1+\frac{1}{n}B(|\nabla w_n|)} = \vartheta T_k\left(\frac{w_{n-1}^q}{|x|^p}\right)+T_k(f) \mbox{ in } \Omega
\end{equation*}
and we prove that $(w_n)_{n\in\mathbb{N}}$ $W_0^{1,p}(\Om)-$converges to $u_k$.
\item We prove that the sequence $(u_k)_{k\in\mathbb{N}}$ of solutions of the family of truncated problem weak converges in $W_0^{1,p}(\Om)$ to a function $u$. This will imply the $W_0^{1,p}(\Om)-$convergence of $T_m(u_k)$ to $T_m(u)$ uniformly on $m$ and finally this will permit to get an a.e.- pointwise convergence of $\nabla u_k$ to $\nabla u$.
\item We prove that $(u_k)_{k\in\mathbb{N}}$ strongly converges to $u$ and moreover that $u$ is the positive weak solution to our problem.
\end{itemize}

Solvability of Problem \eqref{P} already in the $p-$Laplacian case depends of the position of the origin with respect to the domain. This is proved in \cite{DPCalcVar11, MPJMAA12} for $p=2$ and in \cite{APAM03, MMAMPA14} for $p>1$. It is proved in \cite{APAM03} that if the origin belongs to the interior of $\Omega$ we have no
solutions even in entropic sense; for this reason in \cite{MMAMPA14} the authors consider $0 \in \partial \Om$. In this case the existence of solutions turns to be depending on the presence of the source and then on the geometry of the domain.

In this paper we consider $0\in \Omega$ but the attendance of a suitable term involving the gradient, leads a regularizing effect that permits to prove the existence of a solution.\ \\ \ \\

The second part of the paper (developed in Section 4) is devoted to obtain some summability properties of the gradient as well as the second derivatives of the solutions of \eqref{P}. 
In this section we need a further hypothesis on $A(t)$ that is
\begin{equation}\label{1.10}
-1< \inf\limits_{t> 0} \frac{tA'(t)}{A(t)}=: m_A \leq M_A:= \sup\limits_{t> 0} \frac{tA'(t)}{A(t)} <+\infty.
\end{equation}
Therefore the main results of this section can be included in the following
\begin{thm} \label{thm1.3}
 	Assume $1<p<N$ and consider $u\in C^{1,\alpha}_{loc}(\bar{\Omega} \setminus \{0\})\cap C^2(\Om \setminus (Z_u \cup \{0\}))$ a solution to \eqref{ProbPuntuale},
	where $f \in W^{1,\infty}(\bar{\Om})$. We have 	
	
	\begin{equation*}
\int_{E \setminus Z_u} \frac{A(|\nabla u|)|\nabla u_i|^2}{|x-y|^{\gamma}|u_i|^{\beta}} \,dx \leq \mathcal{C} \quad \forall i=1,...,N
	\end{equation*}
	
	\noindent for any $E \Subset \Om \setminus \{0\}$ and uniformly for any $y \in E$, with 
	$$\mathcal{C}:= \mathcal{C}(\gamma,m_A,M_A, \beta, h,\|\nabla u\|_\infty,\rho)$$
	for $0 \leq \beta <1$ and $\gamma < (N-2)$ if $N \geq 3$ ($\gamma=0$ if $N=2$).\\
	
	\noindent Moreover, if we also assume that $\Om$ is a smooth domain and $f$ is nonnegative in $\Om$, we have that 
	
	\begin{equation*}
\int_{\Om \setminus (Z_u \cup \{0\})} \frac{A(|\nabla u|)|\nabla u_i|^2}{|x-y|^{\gamma}|u_i|^{\beta}} \,dx \leq \mathcal{C} \quad \forall i=1,...,N.
	\end{equation*}
\end{thm}\ \\
Here the set $Z_u$ represents the set of critical points of $u$ when $p>2$ or the set of degenerate point for $A$ when $p\in (1,2)$.\ \\
As a corollary of this result we will prove that the Lebesgue measure of $Z_u$ is zero, as in the natural case of $p-$Laplacian.
\begin{thm}\label{thm1.4}
	Let $u\in C^{1,\alpha}(\Omega\setminus \{0\})\cap C^{2}(\Omega\setminus (Z_u \cup \{0\}))$ be a solution to \eqref{ProbPuntuale} with $f\in W^{1,\infty}(\Omega)$ and 
	$f(x)>0$ in $B_{2 \rho}(x_0) \subset \Omega \setminus \{0\}$.  Then 
	\begin{equation*}
		\int_{B_{\rho}(x_0)} \frac{1}{(A(|\nabla u|))^{\alpha r}} \frac{1}{|x-y|^\gamma} dx\le \mathcal{C}
	\end{equation*}
	for any $y\in B_{\rho}(x_0)$, with $\displaystyle \alpha:=\frac{p-1}{p-2}$ if $p>2$ and $\displaystyle \alpha:=\frac{m_A+1}{m_A}$ if $p\in (1,2)$, $r\in (0,1)$, $\gamma < N-2$ if $N \ge 3$, $\gamma=0$ if $N=2$ and $$\mathcal{C}=\mathcal{C}(\gamma, \eta, h, ||\nabla u||_\infty, \rho, x_0, \alpha, M_A,c_2,\tau, \hat{C}).$$

	\noindent If $\Om$ is a smooth domain and $f$ is nonnegative in $\Om$
	\begin{equation*}
	\int_{\tOm} \frac{1}{(A(|\nabla u|))^{\alpha r}} \frac{1}{|x-y|^\gamma}dx \le \mathcal{C}.
	\end{equation*}
	where $\tOm \Subset \Om \setminus \{0\}$ and $y\in \tOm$.
\end{thm}

These results can be framed in the research topic introduced in \cite{DamSciu}; here the authors proved summability properties for the gradient and second derivative for positive solutions of $p-$Laplacian equations. The aim is obtain Sobolev and Poincar\'{e} type inequalities in weighted Sobolev spaces
with weight $A(|\nabla u|)$  and to apply them to the study of monotonicity and symmetry of solutions based on moving plane method.
Since then, these ideas have been exploited, developed and detailed in many papers as \cite{CanDeGS, EspST,MontSciuPM} and much more.
Our results can be easily compared with those in \cite{CanDeGS, EspST}. 

\noindent From Theorem \ref{thm1.3} and \ref{thm1.4}, we obtain this result:

\begin{thm}
	Let $\Om$ be a smooth domain, $u\in C^1(\bar{\Om} \setminus \{0\})$ be a weak solution to \eqref{ProbPuntuale}, $f \in W^{1,\infty}(\bar{\Om})$ and $f$ nonnegative function. Then if $1<p<3$, $u_{i} \in W^{1,2}_{loc}(\Om \setminus \{0\})$, while if $p\geq 3$ then $u_{i} \in W^{1,q}_{loc}(\Om \setminus \{0\})$, for every $i=1,...,N$ and for every $\displaystyle q< \frac{p-1}{p-2}$. Moreover the generalized derivatives of $u_{i}$ coincide with the classical ones, both denoted with $u_{ij}$ almost everywhere in $\Om$.
\end{thm}

\noindent The additional hypothesis \eqref{1.10} is well known and we can find it, for instance, in \cite{CanDeGS, Cia11, Cia14, Cia18,  EspST}. Let us remark that our setting  \eqref{D10}-\eqref{D2}-\eqref{1.10} on $A(t)$ is satisfied, as for example, by $A(t)=t^{p-2}+bt^{q-2}$ with $1<q<p$ and $b>0$; 
this particular choice for the operator $A(t)$ finds application to the study of double-phase equations
$$
-\mbox{div}\left(|\nabla u|^{p-2}\nabla u+b|\nabla u|^{q-2}\nabla u\right)+B(|\nabla u|) = \vartheta \frac{u^q}{|x|^p}+f \qquad \mbox{ in } \Omega.\\
$$

\section{Preliminaries}
In this section, briefly, we enclose some definitions, results and remarks that it will be useful in the rest of the paper. From now on, in order to get a readable notation, generic numerical constant will be denoted by $c$ and will be allowed to vary within a single line or formula. Moreover, Moreover we denote with $f^+:=\max\{f,0\}$ and $f^-:=\min\{f,0\}$.

\begin{rem}
	\noindent With respect to the setting \eqref{crescenzaA}-\eqref{crescitaB'}, we observe that:
	
	\begin{itemize}
		\item 	If $\eta'=0$, the inequality \eqref{D2} becomes
		\begin{align}
		& A( |\eta|)|\eta|^2 \geq c_2|\eta|^p \label{D20}.
		\end{align}
		
		\noindent Moreover, if $\eta,\eta' \in \mathbb{R}^N$, since $ |\eta-\eta '|\leq |\eta|+|\eta '|$, by 
		\eqref{D2} it follows
		\begin{equation}\label{Dp2}
		[A(|\eta|)\eta-A(|\eta'|)\eta']\cdot[\eta-\eta ']\geq c_2 |\eta-\eta '|^p \qquad\mbox{if}\quad p\geq 2.
		\end{equation} 
		
		\item By \eqref{crescenzaA}, we get that 
		\begin{equation} \label{tA lim}
		\lim\limits_{t \rightarrow 0^+} tA(t) < +\infty.
		\end{equation}
		
		\noindent Then, if $0<t<K$ there exists a constant $C_K:=C(K)>0$ such that 
		\begin{equation}\label{tA<K}
		|tA(t)| \leq C_K.
		\end{equation}

		\item  By \eqref{crescenzaA}, \eqref{B0} and \eqref{crescitaB'}, we have that for $\xi \in [0,t]$ 
		\begin{equation}\label{stimaBalto}
		B(t) = \int_0^t B'(s)\,ds \leq  \hat{C}t\xi A(\xi)\leq \hat{C}t^2A(t).
		\end{equation} 
	\end{itemize}
\end{rem}

\vskip 15pt

\vskip 10pt

\begin{defn}
	We say that $u$ is a weak supersolution to Problem \eqref{P}
	if $u\in W_0^{1,p}(\Omega)$ and 
	\begin{equation*}
	\begin{split}
	\int_{\Omega} A(|\nabla u|)(\nabla u, \nabla \varphi )\,dx +\int_{\Om} B(|\nabla u|)\varphi\,dx  & \geq  \, \vartheta \int_{\Om} \frac{u^q}{|x|^p}\varphi\,dx  + \int_{\Om} f \varphi\,dx  \\ & \ \forall
	\varphi \in W_0^{1,p}(\Om) \cap L^{\infty}(\Omega), \  \varphi \geq 0. 
	\end{split}
	\end{equation*}
	We say that $u$ is a weak subsolution
	if $u\in W_0^{1,p}(\Omega)$ and 
	\begin{equation*}
	\begin{split}
	\int_{\Omega} A(|\nabla u|)(\nabla u, \nabla \varphi )\,dx  +\int_{\Om} B(|\nabla u|)\varphi \,dx  &\leq  \ \vartheta \int_{\Om} \frac{u^q}{|x|^p}\varphi \,dx + \int_{\Om} f \varphi \,dx  \\ & \ \forall
	\varphi \in W_0^{1,p}(\Om) \cap L^{\infty}(\Omega), \  \varphi \geq 0. 
	\end{split}
	\end{equation*}
\end{defn}

\noindent A main tool will be the following

\begin{lem}\label{DHardy}(Hardy-Sobolev inequality)
	Suppose $1<p<N$ and $u\in W^{1,p}(\Real ^N)$. Then we have
	$$\int_{\Real^N} \frac{|u|^p}{|x|^p}\,dx \leq C_{N,p}\int_{\Real^N} |\nabla u|^p\,dx $$
	with $C_{N,p}=\left(\frac{p}{N-p}\right)^p$ optimal and not achieved constant.
\end{lem}

\section{Existence of an energy solution}
\subsection{Existence of solutions to the truncated problem.}\ \\
\noindent In this subsection, we are going to study the existence of a solution $u_k \in W_0^{1,p}(\Omega)$ to the truncated problem
\begin{equation}\label{P1_k} \displaystyle
-\mbox{div}(A(|\nabla u_k|)\nabla u_k)+B(|\nabla u_k|) = \vartheta T_k\left(\frac{u_k^q}{|x|^p}\right)+T_k(f) \quad \mbox{ in } \Omega,
\end{equation}

\noindent where 
\begin{equation}\label{deftroncata}
T_k(s)=\max\{\min\{k,s\},-k\}, \quad k>0.
\end{equation}

\noindent We observe that,  since \eqref{B0} holds, $\phi \equiv 0$ is a subsolution to \eqref{P1_k}. \\
\indent By \cite[Theorem 4.1]{BocCroce} and by classical Stampacchia argument, there exists a solution $\psi \in W_0^{1,p}(\Omega) \cap L^{\infty}(\Om)$ to the following problem
\begin{equation}
\begin{cases}\displaystyle
-\mbox{div}(A(|\nabla \psi|)\nabla \psi)= \vartheta k+T_k(f) &\mbox{ in } \Omega,\\
\psi=0 &\mbox{ on } \partial \Omega,
\end{cases}
\end{equation}

\noindent that turns to be a supersolution to the problem \eqref{P1_k}.
In fact, let $\varphi \in W_0^{1,p}(\Om) \cap L^{\infty}(\Omega) \mbox{ and } \varphi \geq 0$, we have
\begin{eqnarray*}
\displaystyle
\int_{\Om} A(|\nabla \psi|)(\nabla \psi , \nabla \varphi)\,dx +\int_{\Om} B(|\nabla \psi|)\varphi\,dx&\geq&\int_{\Om} A(|\nabla \psi|)(\nabla \psi , \nabla \varphi)\,dx\\
= \vartheta \int_{\Om}k\varphi \,dx + \int_{\Om} T_k(f)\varphi \,dx &\geq& \vartheta \int_{\Om}T_k\left(\frac{\psi^q}{|x|^p}\right) \varphi\,dx + \int_{\Om} T_k(f)\varphi\,dx
\end{eqnarray*}

In order to prove the existence of solutions to the problem \eqref{P} is useful to consider a sequence of approximated problems. We take as starting point $w_0=0$ and consider iteratively the problems

\begin{equation}\label{Pniterato}
\begin{cases} \displaystyle
-\mbox{div}\left(A(|\nabla w_n|)\nabla w_n\right)+\frac{B(|\nabla w_n|)}{1+\frac{1}{n}B(|\nabla w_n|)} = \vartheta T_k\left(\frac{w_{n-1}^q}{|x|^p}\right)+T_k(f) &\mbox{ in } \Omega,\\
w_n=0 & \mbox{ on } \partial \Omega.
\end{cases}
\end{equation}

\noindent We observe that $\phi \equiv 0$ and the supersolution $\psi$ to problem \eqref{P1_k} are respectively subsolution and supersolution to the problem \eqref{Pniterato}.

\noindent Then, using the same argument in \cite[Theorem 2.1]{BocMur}, it can be proved the following: 
\begin{prop}\label{Solproblemait}
	There exists $w_n \in W_0^{1,p}(\Omega) \cap L^{\infty}(\Omega)$ solution of the problem \eqref{Pniterato}. Moreover, $0 \leq w_n \leq \psi$ \ $\forall n \in \mathbb{N}$.
\end{prop} 

\noindent Next, we prove the following theorem:
\begin{thm}
	There exists a positive solution to the problem 
	\begin{equation*}\displaystyle
	-\operatorname{div} (A(|\nabla u_k|)\nabla u_k)+B(|\nabla u_k|) = \vartheta T_k\left(\frac{u_k^q}{|x|^p}\right)+T_k(f) \quad \mbox{ in } \Omega.
	\end{equation*} 
\end{thm}

\begin{proof}
	We proceed in two steps.
	\vskip5pt
	\noindent\textit{Step 1:} We show the weak convergence of $\{w_n\}_{n\in \mathbb{N}}$ in $W_0^{1,p}(\Omega)$.\\
	Let us set  $$H_n\left(B(|\nabla w_n|)\right) := \frac{B(|\nabla w_n|)}{1+\frac{1}{n}B(|\nabla w_n|)}. $$
	
	\noindent 
	Considering that $w_n \in W_0^{1,p}(\Omega)\cap L^{\infty}(\Omega)$, we can take $w_n$ as a test function in the approximated problems \eqref{Pniterato} obtaining
	\begin{equation*}
	\begin{split}
	&\int_{\Om} A(|\nabla w_n|)|\nabla w_n|^2\,dx + \int_{\Om}H_n(B(|\nabla w_n|))w_n \,dx \\
	&= \vartheta \int_{\Om} T_k\left(\frac{w_{n-1}^q}{|x|^p}\right)w_n \,dx + \int_{\Om} T_k(f)w_n\,dx \\
	& \leq \vartheta\int_{\Om}kw_n\,dx + \int_{\Om} fw_n \,dx \leq C(k,f,\psi,\vartheta, \Omega)
	\end{split}
	\end{equation*}
	\noindent where $C(k,f,\psi, \vartheta, \Omega)$ is a positive constant.\\
	\noindent Since $H_n(B(|\nabla w_n|))$ is a positive function, from inequality \eqref{D20}, we have that 
	\begin{equation}
	c_2\int_{\Om} |\nabla w_n|^p\,dx \leq \int_{\Om} A(|\nabla w_n|)|\nabla w_n|^2\,dx \leq C(k,f,\psi,\vartheta, \Omega).
	\end{equation}
	
	\noindent Then $\norm{w_n}_{W_0^{1,p}(\Om)}$ is uniformly bounded so, up to a subsequence, we have that 
	\begin{align*}
	&w_n \rightharpoonup \bar{u}_k \ \mbox{ in } W_0^{1,p}(\Omega) \mbox{ and } w_n \rightharpoonup \bar{u}_k \ \mbox{ in }L^{p}(\Omega),\\
	&w_n \shortstack{*\\ [-0.9 ex]$\rightharpoonup$} \  u_k \ \mbox{ in } L^{\infty}(\Omega).
	\end{align*}
	
	\noindent We observe that $u_k= \bar{u}_k$: in fact, by definition of weak  and weak-* convergence we get that 
	
	$$\int_{\Om} w_n \varphi \,dx \ \xrightarrow{n \rightarrow +\infty } \int_{\Om}\bar{u}_k\varphi \,dx \quad \forall \varphi \in L^{p'}(\Omega) $$
	$$\int_{\Om} w_n \varphi \,dx \xrightarrow{n \rightarrow +\infty } \int_{\Om}u_k\varphi \,dx \quad \forall \varphi \in L^1(\Omega) $$
 but, since $p'=\frac{p}{p-1}>1$, the last limit holds also for $\varphi \in L^{p'}(\Om)$, so $u_k= \bar{u}_k$  and  $u_k \in W_0^{1,p}(\Om) \cap L^{\infty}(\Om)$.\\

	\noindent \textit{Step 2:} We show the strong convergence of $\{w_n\}_{n\in \mathbb{N}}$ in $W_0^{1,p}(\Om)$.\\
	To get the strong convergence in $W_0^{1,p}(\Om)$, using Minkowski  inequality, it follows that
	\begin{equation}\label{eq20}
	\norm{w_n-u_k}_{\Wp} \leq \norm{(w_n-u_k)^+}_{\Wp}+\norm{(w_n-u_k)^-}_{\Wp}.
	\end{equation}
\vskip 10pt
	\noindent We first consider the asymptotic behaviour of $\norm{(w_n-u_k)^+}_{\Wp}$.\\
	\noindent Choosing $(w_n-u_k)^+$ as test function in \eqref{Pniterato}, we have
	
	\begin{equation}\label{eq16}
	\begin{split}
	&\int_{\Om} A(|\nabla w_n|)(\nabla w_n, \nabla(w_n-u_k)^+)\,dx + \int_{\Om} H_n(B(|\nabla w_n|))(w_n-u_k)^+ \,dx \\
	& \ \ \ = \vartheta \int_{\Om} T_k\left(\frac{w_{n-1}^q}{|x|^p}\right)(w_n-u_k)^+ \,dx + \int_{\Om} T_k(f)(w_n-u_k)^+\,dx.
	\end{split}
	\end{equation}

	\noindent Since $w_n \rightharpoonup u_k$ in $\Wp$, using the compact Sobolev embedding we obtain $w_n \rightarrow u_k$ in $L^p(\Om)$ con $p\in [1,p^*)$ and then, up to subsequence we obtain that $w_n \rightarrow u_k$ a.e. in $\Om$.  Moreover $(w_n-u_k)^+ \rightarrow 0$ a.e. in $\Omega$. \\
	
	\indent By dominated convergence theorem , the right-hand side of \eqref{eq16} goes to zero when $n$ goes to infinity. Since $\displaystyle \int_{\Om} H_n(B(|\nabla w_n|))(w_n-u_k)^+ \,dx \geq 0$, \eqref{eq16} becomes
	
 \begin{equation}\label{eq19}
  \int_{\Om}A(|\nabla w_n|)(\nabla w_n,\nabla(w_n-u_k)^+)  \,dx \leq o(1) \ \mbox{ as } \  n\rightarrow +\infty.
 \end{equation}
	\\
	\noindent Furthermore, let us observe that $(w_n-u_k)^+ \rightharpoonup 0$ in $\Wp$ implies that 
	
	\begin{equation}\label{eq102}
	\int_{\Om} A(|\nabla u_k|)(\nabla u_k, \nabla(w_n-u_k)^+)\,dx = o(1) \ \mbox{ as } \ n\rightarrow +\infty.
	\end{equation}
	
	\noindent In fact, the operator 
	$$v \mapsto \int_{\Om} A(|\nabla u_k|)(\nabla u_k, \nabla v) \,dx$$ 
	is linear and continuous because, by \eqref{D10} and \eqref{tA<K}, we have

\begin{equation*}
\begin{split}
 &\Bigg|\int_{\Om}  (A(|\nabla u_k|)\nabla u_k, \nabla v) \,dx\Bigg|  \leq \left(\int_{\Om} |A(|\nabla u_k|)\nabla u_k|^{\frac{p}{p-1}}\,dx\right)^{\frac{p-1}{p}} \left(\int_{\Om}|\nabla v|^p\,dx\right)^{\frac{1}{p}} \\
 & \leq \left(C_K^{\frac{p}{p-1}}|\Om \cap \{|\nabla u_k|<K\}| + c\int_{\Om\cap \{|\nabla u|\geq K\}} |\nabla u_k|^p\,dx\right)^{\frac{p-1}{p}}\left(\int_{\Om}|\nabla v|^p\,dx\right)^{\frac{1}{p}} <+\infty. 
\end{split}
\end{equation*}

	\noindent Then, using \eqref{eq102}, we write \eqref{eq19} as
	
	$$\int_{\Om} \left[A(|\nabla w_n|)\nabla w_n-A(|\nabla u_k|)\nabla u_k\right] \cdot  \nabla(w_n-u_k)^+ \,dx\leq o(1). $$
	
	\noindent Using \eqref{D2}, we get 
	\begin{equation*}
	\begin{split}
	0 &\leq c_2 \int_{\Om}  (|\nabla w_n|+|\nabla u_k|)^{p-2}|\nabla(w_n-u_k)^+|^2   \,dx \\
	&\leq \int_{\Om} \left[A(|\nabla w_n|)\nabla w_n-A(|\nabla u_k|)\nabla u_k\right] \cdot  \nabla(w_n-u_k)^+ \,dx \leq o(1), 
	\end{split}
	\end{equation*}

	\noindent hence 
	$$\int_{\Om} \left[A(|\nabla w_n|)\nabla w_n-A(|\nabla u_k|)\nabla u_k\right] \cdot  \nabla(w_n-u_k)^+ \,dx = o(1) \mbox{ as } n\rightarrow +\infty. $$
	
	
	\noindent By \eqref{D2} and \eqref{Dp2}, we obtain
	\begin{equation}\label{eq17}
	\begin{split}
	\displaystyle
	o(1) &=  \int_{\Om} \left[A(|\nabla w_n|)\nabla w_n-A(|\nabla u_k|)\nabla u_k\right] \cdot  \nabla(w_n-u_k)^+ \,dx \\
	&\geq \begin{cases} \displaystyle
	c_2\int_{\Om} \frac{|\nabla(w_n-u_k)^+|^2}{\left(|\nabla w_n|+|\nabla u_k|\right)^{2-p}}\,dx \quad & \mbox{ if } 1<p  < 2,\\
	\displaystyle
	c_2\int_{\Om} |\nabla (w_n-u_k)^+|^p\,dx  & \mbox{ if } p\geq 2. 
	\end{cases}
	\end{split}
	\end{equation}

	\noindent We consider two different case:
	\begin{itemize}
		\item If $p\geq 2$, by \eqref{eq17} we have 
		
		\begin{equation}\label{norma+}
		\norm{(w_n-u_k)^+}_{\Wp} =o(1) \mbox{ as } n\rightarrow +\infty.
		\end{equation}

		\item If $1<p < 2$,
		\begin{equation*}
		\qquad \begin{split}\displaystyle
		\int_{\Om} |\nabla (w_n-u_k)^+|^p\,dx = \displaystyle \bigintssss_{\Om} \frac{|\nabla(w_n-u_k)^+|^p}{(|\nabla w_n|+|\nabla u_k|)^{\frac{p(2-p)}{2}}} (|\nabla w_n|+|\nabla u_k|)^{\frac{p(2-p)}{2}}\,dx.	
		\end{split}
		\end{equation*}
	\noindent By H\"{o}lder inequality and \eqref{eq17}, we obtain 
		
		\begin{equation*}
		\qquad \begin{split}
		&\int_{\Om} |\nabla (w_n-u_k)^+|^p\,dx \\ 
		&\leq \left(\bigintssss_{\Om} \frac{|\nabla(w_n-u_k)^+|^2}{\left(|\nabla w_n|+|\nabla u_k|\right)^{2-p}} \,dx \right)^{\frac{p}{2}} \Bigg(\int_{\Om} (|\nabla w_n|+|\nabla u_k|)^p\,dx\Bigg)^{\frac{2-p}{p}} =o(1).
		\end{split}
		\end{equation*}
	\end{itemize}
Then 
\begin{equation}\label{norma2+}
\norm{(w_n-u_k)^+}_{\Wp}=o(1) \mbox{ as } n\rightarrow +\infty.
\end{equation}

\noindent In order to complete the asymptotic behaviour of \eqref{eq20},  let us consider $e^{-\hat{C}w_n}[(w_n-u_k)^-]$ as test function in \eqref{Pniterato}, where $\hat{C}$ is the positive constant in \eqref{crescitaB'}. 
\begin{equation*}
\begin{split}
&\int_{\Om} A(|\nabla w_n|) (\nabla w_n, \nabla (e^{-\hat{C}w_n}[(w_n-u_k)^-]))\,dx \\
& \ \ \ + \int_{\Om} H_n(B(|\nabla w_n|))e^{-\hat{C}w_n}[(w_n-u_k)^-] \,dx \\
&= \vartheta \int_{\Om} e^{-\hat{C}w_n} T_k\left(\frac{w_{n-1}^q}{|x|^p}\right)(w_n-u_k)^- \,dx + \int_{\Om} e^{-\hat{C}w_n} T_k(f)(w_n-u_k)^-\,dx.
\end{split}
\end{equation*}

\noindent Then 
\begin{equation}\label{eq18}
\begin{split}
&\int_{\Om} e^{-\hat{C}w_n}A(|\nabla w_n|)(\nabla w_n , \nabla (w_n-u_k)^-)\,dx \\
& \ \ \ + \int_{\Om} e^{-\hat{C}w_n} \left[H_n(B(|\nabla w_n|))-\hat{C}A(|\nabla w_n|)|\nabla w_n|^2\right](w_n-u_k)^-\,dx\\
&= \vartheta \int_{\Om} e^{-\hat{C}w_n} T_k\left(\frac{w_{n-1}^q}{|x|^p}\right)(w_n-u_k)^- \,dx + \int_{\Om} e^{-\hat{C}w_n} T_k(f)(w_n-u_k)^-\,dx.
\end{split}
\end{equation}

\noindent By \eqref{stimaBalto}, we have that 
\begin{equation*}
H_n(B(|\nabla w_n|))-\hat{C}A(|\nabla w_n|)|\nabla w_n|^2 \leq B(|\nabla w_n|)-\hat{C}A(|\nabla w_n|)|\nabla w_n|^2 \leq 0 
\end{equation*}

\noindent then 
\begin{equation*}
\int_{\Om} e^{-\hat{C}w_n} \left[H_n(B(|\nabla w_n|))-\hat{C}A(|\nabla w_n|)|\nabla w_n|^2\right](w_n-u_k)^-\,dx  \geq 0.
\end{equation*}

\noindent As above, since $(w_n-u_k)^- \rightarrow 0$, by dominated convergence theorem, the right-hand side of \eqref{eq18} tends to zero when $n$ goes to infinity.\\

\noindent So, the equation \eqref{eq18} becomes 
 \begin{equation}\label{e105}
 \int_{\Om} e^{-\hat{C}w_n}A(|\nabla w_n|)(\nabla w_n , \nabla (w_n-u_k)^-) \,dx \leq o(1).
 \end{equation}
 
 \noindent Since $(w_n-u_k)^-\rightharpoonup 0$ for $n \rightarrow +\infty$, 
 \begin{equation}\label{eq106}
 \begin{split}
 &\int_{\Om} A(|\nabla w_n|)(\nabla w_n , \nabla (w_n-u_k)^-) \,dx  \\
 &= \int_{\Om} (A(|\nabla w_n|)\nabla w_n - A(|\nabla u_k|)\nabla u_k , \nabla (w_n-u_k)^-) \,dx \\
 & \ \ \ + \int_{\Om} A(|\nabla u_k|)(\nabla u_k , \nabla (w_n-u_k)^-) \,dx \\ 
 &= \int_{\Om} e^{\hat{C}w_n}e^{-\hat{C}w_n}(A(|\nabla w_n|)\nabla w_n - A(|\nabla u_k|)\nabla u_k , \nabla (w_n-u_k)^-) \,dx + o(1).
 \end{split}
 \end{equation}

\noindent By Proposition \ref{Solproblemait}, $0\leq w_n \leq \psi$, therefore one has that $e^{\hat{C}w_n} \leq e^{\hat{C}\psi}=: \gamma >0$ uniformly on $n$ being $\psi \in L^{\infty}(\Omega)$. Then, using \eqref{D2} and \eqref{e105}, we estimate \eqref{eq106} as  

\begin{equation}\label{eq107}
\begin{split}
&\int_{\Om} A(|\nabla w_n|)(\nabla w_n , \nabla (w_n-u_k)^-) \,dx  \\
&= \int_{\Om} e^{\hat{C}w_n}e^{-\hat{C}w_n}(A(|\nabla w_n|)\nabla w_n - A(|\nabla u_k|)\nabla u_k , \nabla (w_n-u_k)^-) \,dx + o(1)\\
&\leq  \gamma \int_{\Om} e^{-\hat{C}w_n}(A(|\nabla w_n|)\nabla w_n - A(|\nabla u_k|)\nabla u_k , \nabla (w_n-u_k)^-) \,dx + o(1)\\
& \leq -\gamma \int_{\Om} e^{-\hat{C}w_n} A(|\nabla u_k|)(\nabla u_k , \nabla (w_n-u_k)^-) \,dx + o(1).
\end{split}
\end{equation}

\noindent Now, we estimate the right hand-side of \eqref{eq107} in this way: 

\begin{equation}\label{eq900}
\begin{split}
 &\int_{\Om} e^{-\hat{C}w_n} A(|\nabla u_k|)(\nabla u_k , \nabla (w_n-u_k)^-) \,dx\\ 
 &=  \int_{\Om} e^{-\hat{C}(w_n-u_k)^-} e^{-\hat{C}u_k}A(|\nabla u_k|)(\nabla u_k , \nabla (w_n-u_k)^-) \,dx
 \\
 &= -\frac{1}{\hat{C}}  \int_{\Om}e^{-\hat{C}u_k}A(|\nabla u_k|)(\nabla u_k , \nabla (e^{-\hat{C}(w_n-u_k)^-}-1)) \,dx
 \end{split}
\end{equation}

\noindent Since $\{ e^{-\hat{C}(w_n-u_k)^-}-1\}_{n\in \mathbb{N}}$ is uniformly bounded in $W_0^{1,p}(\Om)$, up to a subsequence there exists $g\in W_0^{1,p}(\Om)$ such that

\begin{equation}\label{convg}
e^{-\hat{C}(w_n-u_k)^-}-1  \rightharpoonup g \mbox{ as } n\rightarrow +\infty.
\end{equation}

\noindent In particular we note that $g \equiv 0$ since $w_n \rightarrow u_k$ a.e. in $\Om$ as $n\rightarrow +\infty.$

\noindent Hence, by \eqref{convg}, for $n\rightarrow +\infty$

\begin{equation}\label{eq108}
\int_{\Om}e^{-\hat{C}u_k}A(|\nabla u_k|)(\nabla u_k , \nabla ( e^{-\hat{C}(w_n-u_k)^-}-1)) \,dx \rightarrow 0. 
\end{equation}

\noindent Then, using \eqref{eq900} and \eqref{eq108} in \eqref{eq107}, one has
\begin{equation*}
\int_{\Om} A(|\nabla w_n|)(\nabla w_n , \nabla (w_n-u_k)^-) \,dx \leq o(1).
\end{equation*}

\noindent Arguing in the same way as we have done from equation \eqref{eq19} to \eqref{norma2+}, we obtain 
that 
\begin{equation}\label{norma-}
\norm{(w_n-u_k)^-}_{\Wp} \rightarrow 0 \mbox{ as } n\rightarrow +\infty.
\end{equation}

\noindent Summarizing up, from \eqref{eq20} we get that 
 \begin{equation}\label{convnorma}
\norm{w_n-u_k}_{\Wp} \rightarrow 0 \mbox{ as } n\rightarrow +\infty
 \end{equation}

\noindent Up to a subsequence, $|\nabla w_n| \rightarrow |\nabla u_k|$  a.e. in $\Omega$ and there exists $v(x)\in L^1(\Omega)$ such that $|\nabla w_n|^p \leq v(x)$ a.e. in $\Omega$. 

\noindent Since $B$ is continuous, we get that 
$$\lim\limits_{n \rightarrow +\infty} H_n(B(|\nabla w_n|)) = \lim\limits_{n \rightarrow +\infty} \frac{B(|\nabla w_n|)}{1+\frac{1}{n}B(|\nabla w_n|)} = B(|\nabla u_k|) \mbox{ a.e. in } \Omega.$$

\noindent Moreover, if $E \subset \Omega$ is a measurable set, by \eqref{stimaBalto}, \eqref{D10} and \eqref{tA<K}, we have 
\begin{equation*}
\begin{split}
\int_{E} H_n(B(|\nabla w_n|))\,dx & \leq \int_{E} B(|\nabla w_n|) \,dx \leq \hat{C}\int_{E} |A(|\nabla w_n|)||\nabla w_n|^2 \,dx \\
& \leq KC_K\hat{C}|E\cap \{|\nabla w_n| < K\}| + c_1\hat{C} \int_{E\cap \{|\nabla w_n| \geq K\}} |\nabla w_n|^p\,dx \\
& \leq KC_K\hat{C}|E| + c_1\hat{C} \int_E v(x)\,dx
\end{split}
\end{equation*}

\noindent and, since $v(x)\in L^1(\Omega)$, we get that $H_n(B(|\nabla w_n|))$ is uniformly integrable.\\

\noindent Using Vitali Theorem we have
$$ H_n(B(|\nabla w_n|)) \rightarrow B(|\nabla u_k|)  \mbox{ in } L^1(\Omega).$$

To conclude the proof, passing to the limit in weak formulation of problem \eqref{Pniterato}, we obtain that $u_k \in \Wp \cap L^{\infty}(\Omega)$ satisfies 
\begin{equation*}
\begin{split}
\int_{\Om} & A(|\nabla u_k|)(\nabla u_k,\nabla \varphi)\,dx \  +\int_{\Om} B(|\nabla u_k|)\varphi \,dx \\
& \ =\vartheta \int_{\Om} T_k\left(\frac{u_k^q}{|x|^p}\right)\varphi\,dx +\int_{\Om} T_k(f)\varphi \,dx \quad \forall \varphi \in \Wp \cap L^{\infty}(\Omega),
\end{split}
\end{equation*}

\noindent hence $u_k$ is a positive weak solution to the problem \eqref{P1_k}.
\end{proof}

\subsection{Existence of solutions to the problem \eqref{P}.}\ \\
\noindent In order to prove the existence of solution to the problem \eqref{P}, we exploit the following Lemma (see \cite{MontSciuPM}). 

\begin{lem}\label{lemmastima}
	Let $\psi_n(s)$ defined as
	\begin{equation}\label{defpsi1}
	\psi_n(s)= \int_0^s T_n(t)^{\frac{1}{p}}\,dt.
	\end{equation}
	
	\noindent For fixed $q \in [p-1, p), \forall \varepsilon >0$ and $\forall n>0$, there exists $C_\varepsilon$ such that
	\begin{equation}\label{eqstima}
	s^qT_n(s) \leq \varepsilon \psi_n^p(s) +C_{\varepsilon}, \qquad s\geq 0.
	\end{equation}
\end{lem}

\noindent Next we prove our main Theorem \ref{ExistenceM}.\ \\ \ \\
 {\bf Theorem 1.2.}
	\emph{Let $f \in L^1(\Omega)$ a positive function; then for every $\vartheta >0$ there exists a weak solution $u \in \Wp $ to \eqref{P}.}

\begin{proof}Again our proof is divided in steps.\ \\
	\textit{Step 1:} We show the weak convergence of $\{u_k\}_{k\in \mathbb{N}}$ in $\Wp$.\\
	Taking as function test $T_n(u_k)$ in the truncated problem \eqref{P1_k}, we obtain
	\begin{equation*}
	\begin{split}
	\int_{\Om} &A(|\nabla u_k|)(\nabla u_k , \nabla T_n(u_k)) \,dx + \int_{\Om} B(|\nabla u_k|)T_n(u_k)\,dx \\
	& = \vartheta \int_{\Om} T_k\left(\frac{u_k^q}{|x|^p}\right)T_n(u_k)\,dx +\int_{\Om} T_k(f)T_n(u_k)\,dx.
	\end{split}
	\end{equation*}
\noindent By \eqref{defpsi1} we have that $|\nabla \psi_n(u_k)|^p= T_n(u_k)|\nabla u_k|^p$ and since $u_k$ is a positive function, we get $$T_k\left(\frac{u_k^q}{|x|^p}\right) \leq \frac{u_k^q}{|x|^p}.$$

\noindent Then
\begin{equation}\label{eq21}
\begin{split}
\int_{\Om}  A(|& \nabla u_k|) (\nabla u_k , \nabla T_n(u_k)) \,dx \ + \ \int_{\Om} B(|\nabla u_k|)T_n(u_k)\,dx \\
& \ \ \ \ \ \leq \vartheta \int_{\Om} \frac{u_k^q}{|x|^p}T_n(u_k) \,dx +n\int_{\Om} f \,dx.
\end{split}
\end{equation}

\noindent Using Lemma \ref{lemmastima} and Lemma \ref{DHardy}, the equation \eqref{eq21} becomes
\begin{equation}\label{e33}
\begin{split}
&\int_{\Om} A(|\nabla T_{n}(u_k)|)|\nabla T_{n}(u_k)|^2 \,dx \ + \ \int_{\Om} B(|\nabla u_k|)T_n(u_k)\,dx \\
&\leq \vartheta \int_{\Om} (\varepsilon \psi^p_n(u_k)+C_{\varepsilon})\frac{1}{|x|^p}\,dx + n\norm{f}_1\\
& \leq \vartheta \varepsilon C_{N,p} \int_{\Om} |\nabla \psi_n(u_k)|^p\,dx
+\vartheta C_{\varepsilon}\int_{\Om} \frac{1}{|x|^p}\,dx + n\norm{f}_1.
\end{split}
\end{equation}
\\
Using \eqref{stimaBbasso} and \eqref{D20}, equation \eqref{e33} becomes
\begin{equation*}
\begin{split}
&c_2\int_{\Om} |\nabla T_{n}(u_k)|^p \,dx \ + \sigma\int_{\Om} |\nabla \psi_n(u_k)|^p\,dx\\
& \leq \vartheta \varepsilon C_{N,p} \int_{\Om} |\nabla \psi_n(u_k)|^p\,dx
+\vartheta C_{\varepsilon}\int_{\Om} \frac{1}{|x|^p}\,dx + n\norm{f}_1
\end{split}
\end{equation*}

\noindent and since $\Omega$ is a bounded set and $p<N$ we get
\begin{equation*}
\begin{split}
c_2\int_{\Om} |\nabla T_{n}(u_k)|^p \,dx \ + \ (\sigma- \vartheta \varepsilon C_{N,p}) \int_{\Om} |\nabla \psi_n(u_k)|^p\,dx \leq C(n, \varepsilon,p, f, \vartheta, \Omega).
\end{split}
\end{equation*}
\\

\noindent If we choose $\displaystyle  0< \varepsilon < \frac{\sigma}{\vartheta C_{N,p}}$, by setting $\tilde{C} :=  \sigma-\vartheta\varepsilon C_{N,p}> 0$, we get

\begin{equation}\label{stima3}
\begin{split}
\int_{\Om} |\nabla T_{n}(u_k)|^p \,dx \ + \ \frac{\tilde{C}}{c_2} \int_{\Om} |\nabla \psi_n(u_k)|^p\,dx \leq C(n,f, p, c_2, \sigma, \vartheta, C_{N,p}, \Omega).
\end{split}
\end{equation}
\\
\noindent Fixed  $\displaystyle l \geq \frac{c_2}{\tilde{C}}$, since $|\nabla \psi_l(u_k)|^p= T_l(u_k)|\nabla u_k|^p= l |\nabla u_k|^p$ on $\Om \cap \{u_k \geq l\}$ and using \eqref{stima3}, one has
\begin{equation} \label{u_kunifb}
\begin{split}
\int_{\Om} |\nabla u_k|^p\,dx & = \int_{\Om \cap \{u_k \leq l\}} |\nabla u_k|^p\,dx + \int_ {\Om \cap \{u_k \geq l\}} |\nabla u_k|^p\,dx\\
&= \int_{\Om} |\nabla T_l(u_k)|^p\,dx + \frac{1}{l}\int_ {\Om \cap \{u_k \geq l\}} |\nabla \psi_l(u_k)|^p\,dx\\
& \leq \int_{\Om} |\nabla T_l(u_k)|^p\,dx + \frac{1}{l}\int_ {\Om} |\nabla \psi_l(u_k)|^p\,dx \\
& \leq \int_{\Om} |\nabla T_{l}(u_k)|^p \,dx \ + \ \frac{\tilde{C}}{c_2} \int_{\Om} |\nabla \psi_l(u_k)|^p\,dx \\ & \leq C(l,f, p, c_2, \sigma, \vartheta, C_{N,p}, \Omega).
\end{split}
\end{equation}

\noindent Then, $\norm{u_k}_{\Wp}$ is uniformly bounded on $k$. Therefore, up to a subsequence, it follows that $u_k\rightharpoonup u$ in $\Wp$ and a.e. in $\Omega$.\\

\noindent \textit{Step 2:} We show the strong convergence in $L^1(\Omega)$ of the singular term.\\
\noindent By H\"{o}lder and Hardy inequalities we have
\begin{equation*}
\begin{split}
&\int_{\Om}T_k\left(\frac{u_k^q}{|x|^p}\right)\,dx \leq \int_{\Om} \frac{u_k^q}{|x|^p}\,dx = \int_{\Om} \frac{u_k^q}{|x|^q}\frac{1}{|x|^{p-q}}\,dx \\
& \leq \left( \int_{\Om}  \frac{u_k^p}{|x|^p}\,dx \right)^{\frac{q}{p}} \left(\int_{\Om}\frac{1}{|x|^{p}} \,dx\right)^{\frac{p-q}{p}}\leq c  \left( \int_{\Om}  |\nabla u_k|^p\,dx\, \right)^{\frac{q}{p}} \leq \bar{C},
\end{split}
\end{equation*}

\noindent where $\bar{C}$ is a positive constant that does not depend on $k$. Hence we deduce that $T_k\left(\frac{u_k^q}{|x|^p}\right)$ is bounded in $L^1(\Om)$.\\
\\
 Since $u_k \rightarrow u$ a.e. in $\Om$, by definition of $T_k(s)$, we get that 

$$T_k\left(\frac{u_k^q}{|x|^p}\right) \rightarrow \frac{u^q}{|x|^p} \mbox{ a.e. in } \Om.$$

\noindent Noting that $T_k\left(\frac{u_k^q}{|x|^p}\right) \geq 0$, we can use Fatou Lemma and obtain

$$\int_{\Om} \frac{u^q}{|x|^p}\, dx \leq \liminf\limits_{k \rightarrow +\infty} \int_{\Om} T_k\left(\frac{u_k^q}{|x|^p}\right)\,dx, $$

\noindent then $\displaystyle\frac{u^q}{|x|^p} \in L^1(\Om)$.
\vskip 2pt
\noindent We now consider a measurable set $E \subset \Omega$. Using again Fatou Lemma and H\"{o}lder inequality we have

\begin{equation*}
\begin{split}
&\int_{E} T_k\left(\frac{u_k^q}{|x|^p}\right)\,dx \leq \int_{E} \frac{u_k^q}{|x|^p} \,dx  \leq \liminf\limits_{n \rightarrow +\infty} \int_{E} \frac{w_n^q}{|x|^p}\,dx \leq \int_{E} \frac{\psi^q}{|x|^p}\,dx\\
& = \int_{E} \frac{\psi^q}{|x|^q}\frac{1}{|x|^{p-q}}\,dx \leq   \left( \int_{E}  \frac{\psi^p}{|x|^p}\, \,dx \right)^{\frac{q}{p}} \left(\int_{E }\frac{1}{|x|^{p}}\, dx\right)^{\frac{p-q}{p}} \\
& \leq c \norm{\nabla \psi}_p^{\frac{q}{p}} \left(\int_{E }\frac{1}{|x|^{p}}\, dx\right)^{\frac{p-q}{p}} < \delta(\mu(E))
\end{split}
\end{equation*}

\noindent uniformly on $k$, where $\lim\limits_{s \rightarrow 0} \delta(s)=0$.\\
Thus, from Vitali Theorem, we get
\begin{equation}\label{conL1T_k}
T_k\left(\frac{u_k^q}{|x|^p}\right) \rightarrow \frac{u^q}{|x|^p} \mbox{ in } L^1(\Om).
\end{equation}

\vskip 10pt
\noindent \textit{Step 3:} We show the strong convergence of $|\nabla u_k|^p \rightarrow |\nabla u|^p$ in $L^1(\Om)$.

\noindent We need two preliminary results.

\begin{lem}\label{lim0}
	Let $u_k$ be a weak solution to the problem \eqref{P1_k}. Then 
	
\begin{equation*}
\lim\limits_{n \rightarrow +\infty} \int_{\{u_k \geq n\}} |\nabla u_k|^p \,dx=0
\end{equation*}
uniformly on $k$.

\begin{proof}
	Let us consider the functions
	$$G_n(s)= s-T_n(s) \quad  \mbox{ and } \quad \tau_{n-1}(s)=T_1(G_{n-1}(s)).$$
	
	\noindent We get that 
	\begin{equation}\label{defpsin-1}
	\tau_{n-1}(u_k)=T_1(u_k-T_{n-1}(u_k))=\begin{cases}
	0 & \mbox{if } u_k <n-1,\\
	u_k-(n-1) & \mbox{if } n-1 \leq u_k < n,\\
	1 & \mbox{if } u_k \geq n.
	\end{cases} 
	\end{equation}
	
	\noindent Using $\tau_{n-1}(u_k)$ as a function test in \eqref{P1_k}, we get
	\begin{equation}\label{eq25}
	\begin{split}
	& \int_{\Om} A(|\nabla u_k|)(\nabla u_k, \nabla \tau_{n-1}(u_k))\,dx +  \int_{\Om} B(|\nabla u_k|)\tau_{n-1}(u_k)\,dx \\
	&= \vartheta \int_{\Om} T_k\left(\frac{u_k^q}{|x|^p}\right)\tau_{n-1}(u_k)\,dx + \int_{\Om} T_k(f)\tau_{n-1}(u_k)\,dx.
	\end{split}
	\end{equation}

     \noindent We want to estimate the left-hand side of \eqref{eq25}. Using the definition of $\tau_{n-1}(u_k)$ and \eqref{D20}, we get
	\begin{equation}\label{eq23}
	\begin{split}
	&\int_{\Om} A(|\nabla u_k|)(\nabla u_k, \nabla \tau_{n-1}(u_k))\,dx = \int_{\{n-1 \leq u_k < n\}} A(|\nabla u_k|)(\nabla u_k, \nabla u_k)\,dx \\
	&=\int_{\{n-1 \leq u_k < n\}} A(|\nabla u_k|)|\nabla u_k|^2\,dx \geq c_2 \int_{\{n-1 \leq u_k < n\}} |\nabla u_k|^p\,dx \\
	&= c_2 \int_{\Om} |\nabla \tau_{n-1}(u_k)|^p\,dx
	\end{split}
	\end{equation}
	while from \eqref{stimaBbasso} one has
	\begin{equation}\label{eq24}
  \int_{\Om} B(|\nabla u_k|)\tau_{n-1}(u_k)\,dx \geq 	\sigma \int_{\Om}|\nabla u_k|^p\tau_{n-1}(u_k)\,dx.
	\end{equation}
	
	\noindent Then, using \eqref{eq23} and \eqref{eq24} in \eqref{eq25}, one has 
	
	\begin{equation}\label{e103}
	\begin{split}
	c_2 \int_{\Om} |\nabla \tau_{n-1}(u_k)|^p\,dx  &+ \sigma \int_{\Om}|\nabla u_k|^p\tau_{n-1}(u_k)\,dx \\
	&\leq \vartheta \int_{\Om} T_k\left(\frac{u_k^q}{|x|^p}\right)\tau_{n-1}(u_k)\,dx + \int_{\Om} T_k(f)\tau_{n-1}(u_k)\,dx.
	\end{split}
	\end{equation}
	
	\noindent Moreover, from \eqref{defpsin-1} 
	\begin{equation}\label{eq22}
	\tau_{n-1}(u_k)|\nabla u_k|^p \geq |\nabla u_k|^p \chi_{\{u_k \geq n\}},
	\end{equation}
	
	\noindent hence \eqref{e103} becomes 
	\begin{equation}\label{e104}
	\begin{split}
	\sigma\int_{\{u_k \geq n\}} |\nabla u_k|^p\,dx  &\leq c_2 \int_{\Om} |\nabla \tau_{n-1}(u_k)|^p\,dx  + \sigma \int_{\Om}|\nabla u_k|^p\tau_{n-1}(u_k)\,dx \\
	&\leq \vartheta \int_{\Om} T_k\left(\frac{u_k^q}{|x|^p}\right)\tau_{n-1}(u_k)\,dx + \int_{\Om} T_k(f)\tau_{n-1}(u_k)\,dx.
	\end{split}
	\end{equation}

	\noindent By \eqref{u_kunifb}, $\{u_k\}_{k\in \mathbb{N}}$ is uniformly bounded in $\Wp$ on $k$ so, up to a subsequence, $u_k$ weakly converges in $\Wp$, strongly converges in $L^p(\Om)$ with $1\leq p < p^*$ and a.e. in $\Om$.\\ Then 
	\begin{equation}\label{misuranulla}
	\begin{split}
	&|\{x \in \Om: n-1 \leq u_k(x) < n \} | \rightarrow 0 \ \mbox{ if } n\rightarrow +\infty,\\
	&|\{x \in \Om: u_k(x) \geq  n \} | \rightarrow 0 \qquad \qquad \mbox{ if }
	n\rightarrow +\infty
	\end{split}
	\end{equation}
	
	\noindent uniformly on $k$.\\
	\noindent By \eqref{defpsin-1} and \eqref{e104} we get
	
	\begin{equation*}
	\begin{split}
	& \sigma\int_{\{u_k \geq n\}} |\nabla u_k|^p\,dx  \leq \vartheta \int_{\{n-1 \leq u_k < n\}} T_k\left(\frac{u_k^q}{|x|^p}\right)\tau_{n-1}(u_k)\,dx \\
	& \ +  \vartheta \int_{\{u_k \geq n\}} T_k\left(\frac{u_k^q}{|x|^p}\right)\tau_{n-1}(u_k)\,dx+ \int_{\{n-1 \leq u_k < n\}} T_k(f)\tau_{n-1}(u_k)\,dx \\
	& \ + \int_{\{u_k \geq n\}}T_k(f)\tau_{n-1}(u_k)\,dx.
	\end{split}
	\end{equation*}
	
	\noindent Then, by \eqref{misuranulla} and dominated convergence theorem, we have
	
	\begin{equation*}
	\lim\limits_{n \rightarrow +\infty} \int_{\{u_k \geq n\}} |\nabla u_k|^p\,dx =0.
	\end{equation*}
\end{proof}

\end{lem}
\begin{lem}\label{lemmaconvTm}
	Consider $u_k\rightharpoonup u \mbox{ in } \Wp$. Then one has for every m
	$$T_m(u_k) \rightarrow T_m(u) \mbox{ in } \Wp \quad \mbox{ for }  k\rightarrow +\infty.$$
\end{lem}

\begin{proof}
	Notice that	
	\begin{equation}\label{normTm}
	\begin{split}
	\norm{T_m(u_k)-T_m(u)}_{\Wp} &\leq \norm{(T_m(u_k)-T_m(u))^+}_{\Wp}\\
	& \ \ +\norm{(T_m(u_k)-T_m(u))^-}_{\Wp}.
	\end{split}
	\end{equation}
	\noindent First of all we study the asymptotic behaviour of $\norm{(T_m(u_k)-T_m(u))^+}_{\Wp}$.\\ If we take $(T_m(u_k)-T_m(u))^+$ as a test function in \eqref{P1_k} we get
	\begin{equation}\label{eq26}
	\begin{split}
	&\int_{\Om} A(|\nabla u_k|)(\nabla u_k, \nabla (T_m(u_k)-T_m(u))^+)\,dx +\int_{\Om} B(|\nabla u_k|)(T_m(u_k)-T_m(u))^+\,dx \\
	&= \int_{\Om} \left(\vartheta T_k\left(\frac{u_k^q}{|x|^p}\right) +T_k(f) \right)(T_m(u_k)-T_m(u))^+\,dx.
	\end{split}
	\end{equation}
	\noindent Since $T_m(u_k) \rightharpoonup T_m(u)$ in $\Wp$ and $T_m(u_k) \rightarrow T_m(u)$ a.e. in $\Omega$, we have $(T_m(u_k)-T_m(u))^+ \rightarrow 0$ a.e. in $\Om$ and $(T_m(u_k)-T_m(u))^+\rightharpoonup 0$ in $\Wp$. Thus, the right- hand side of \eqref{eq26}, using \eqref{conL1T_k} and dominated convergence theorem, tends to zero as $k$ goes to infinity. \\ Then, for $ k\rightarrow +\infty$
	\begin{equation}\label{eq27}
	\begin{split}
	&\int_{\Om} A(|\nabla u_k|)(\nabla u_k, \nabla (T_m(u_k)-T_m(u))^+)\,dx \\
	& \ + \int_{\Om} B(|\nabla u_k|)(T_m(u_k)-T_m(u))^+\,dx = o(1).
	\end{split}
	\end{equation}
	
	\noindent Let us note that equation \eqref{eq27} is equivalent to 
	\begin{equation}\label{eq28}
	\begin{split}
	&\int_{\Om} \big(A(|\nabla u_k|)\nabla u_k-A(|\nabla u|)\nabla u, \nabla (T_m(u_k)-T_m(u))^+\big)\,dx \\
	& \ \ \ + \int_{\Om} B(|\nabla u_k|)(T_m(u_k)-T_m(u))^+\,dx \\
	&= -\int_{\Om}A(|\nabla u|)(\nabla u, \nabla(T_m(u_k)-T_m(u))^+)\,dx +  o(1) =o(1),
	\end{split}
	\end{equation}
	since by $(T_m(u_k)-T_m(u))^+\rightharpoonup 0$ in $\Wp$ one has that
	
	\begin{equation*}
	\int_{\Om}A(|\nabla u|)(\nabla u, \nabla(T_m(u_k)-T_m(u))^+)\,dx =o(1) \mbox{ as } k\rightarrow +\infty.
	\end{equation*}
	
	\noindent We have that 
	\begin{equation}\label{gradTm+}
	\qquad \qquad (T_m(u_k)-T_m(u))^+=
	\begin{cases}
	(u_k - T_m(u))^+  &\mbox{ if } 0< u_k \leq m,\\
	m - T_m(u)  & \mbox{ if } u_k > m, 
	\end{cases}
	\end{equation}
	$$ \ \ \nabla (T_m(u_k)-T_m(u))^+ =\begin{cases}
	\nabla(u_k - u)^+  &\mbox{ if } u_k \leq m, u\leq m,\\
	0  &\mbox{ if } u_k \leq m, u>m,\\
	0  & \mbox{ if } u_k > m, u\geq m, \\
	- \nabla u  & \mbox{ if } u_k > m, u<m. 
	\end{cases}$$
	
	\noindent So, using \eqref{gradTm+} and the fact that $B$ is a positive function, \eqref{eq28} becomes
	\begin{equation}\label{eq29}
	\begin{split}
	& \int_{\Om \cap \{u_k\leq m, \  u\leq m\}} \Big(A(|\nabla u_k|)\nabla u_k-A(|\nabla u|)\nabla u, \nabla (u_k-u)^+\Big)\,dx   \\
	& \ \ \ + \int_ {\Om \cap \{u_k> m, u< m\}} \Big(A(|\nabla u_k|)\nabla u_k-A(|\nabla u|)\nabla u, -\nabla u\Big)\,dx \\
	&=\int_{\Om} \Big(A(|\nabla u_k|)\nabla u_k-A(|\nabla u|)\nabla u, \nabla (T_m(u_k)-T_m(u))^+\Big) \,dx \leq o(1).
	\end{split}
	\end{equation}
	\\
	
	\noindent We set:
	\begin{align*}\displaystyle
	&D:= \int_{\Om} \Big(A(|\nabla u_k|)\nabla u_k-A(|\nabla u|)\nabla u, \nabla (T_m(u_k)-T_m(u))^+\Big) \,dx,\\
	&D_1:= \int_{\Om \cap \{u_k\leq m, \  u\leq m\}} \Big(A(|\nabla u_k|)\nabla u_k-A(|\nabla u|)\nabla u, \nabla (u_k-u)^+\Big)\,dx, \\
	&D_2:=  \int_ {\Om \cap \{u_k> m, u< m\}} \Big(A(|\nabla u_k|)\nabla u_k-A(|\nabla u|)\nabla u, -\nabla u\Big)\,dx.  
	\end{align*}
	
	\begin{itemize}
		\item 
		
		For the term $D_1$, using \eqref{D2}, we have
		\begin{equation}\label{eq30}
		\begin{split}\displaystyle
		\qquad \quad D_1&= \int_{\Om \cap \{u_k\leq m, \  u\leq m\}} \Big(A(|\nabla u_k|)\nabla u_k-A(|\nabla u|)\nabla u, \nabla (u_k-u)^+\Big)\,dx \\
		&\geq c_2 \int_{\Om \cap \{u_k\leq m, \  u\leq m\}} (|\nabla u_k|+|\nabla u|)^{p-2}|\nabla(u_k-u)^+|^2 \,dx \geq 0.
		\end{split}
		\end{equation}
		
		\item 
		We write $D_2$ as 
		
		\begin{equation*}
		\begin{split}
		D_2&=D_{2,1}+D_{2,2}\\
		&= \int_{\Om \cap \{u_k> m, u< m\}} A(|\nabla u|)|\nabla u|^2\,dx \\
		& \ \ \ +\int_ {\Om \cap \{u_k> m, u< m\}} A(|\nabla u_k|) \left(\nabla u_k, -\nabla u\right)\,dx.
		\end{split}
		\end{equation*}
		
		\noindent Then, by \eqref{tA<K} and \eqref{D10}, we have
		\begin{equation*}
		\begin{split}
		|D_{2,1}| & \leq \int_{\Om \cap \{u_k> m, u< m\}} |A(|\nabla u|)||\nabla u|^2\,dx \\
		& \leq KC_K \displaystyle \left|\Om \cap \{u_k> m, u< m\} \cap \{|\nabla u| < K\} \right| \\
		& \ \ \ + c_1 \int_{\Om} |\nabla u|^p \chi_{\Om \cap \{u_k> m, u< m\} \cap \{|\nabla u| \geq K\}} \,dx.
		\end{split}
		\end{equation*}
		
		\noindent Since $u_k \rightarrow u$ a.e. in $\Om$ for $k\rightarrow +\infty$, we get that $\chi_{\Om \cap \{u_k> m, \ u< m\}} \rightarrow 0$. Moreover, $|\nabla u|^p\in L^1(\Om)$, then using Dominated Convergence, one has that 
		
		$$|D_{2,1}| =o(1). $$
		\vskip 15pt
		\noindent Analogously, using H\"{o}lder inequality, \eqref{tA<K}, \eqref{D10}, dominated convergence theorem and \eqref{u_kunifb}
		
		\begin{equation*}
		\begin{split}
		&|D_{2,2}|= \Bigg| \int_{\Om \cap \{u_k> m, \ u< m\}} \left( A(|\nabla u_k|)\nabla u_k, \nabla u \right) \,dx \Bigg|  \\
		&\leq \left(\int_{\Om} |A(|\nabla u_k|)\nabla u_k|^{\frac{p}{p-1}}\,dx\right)^{\frac{p-1}{p}}\left(\int_{\Om \cap \{u_k> m, \ u< m\}}|\nabla u|^p\,dx\right)^{\frac{1}{p}} \\
		&  \leq \left( C_K^{\frac{p}{p-1}}|\Om \cap \{|\nabla u_k| < K\}| + c\int_{\Om\cap \{|\nabla u_k| \geq K\}} |\nabla u_k|^p\,dx\right)^{\frac{p-1}{p}}\\
		& \ \ \ \cdot \left(\int_{\Om \cap \{u_k> m, \ u< m\}}|\nabla u|^p\,dx\right)^{\frac{1}{p}}
		\end{split}
		\end{equation*}
		
		\noindent Arguing for that last integral as in $|D_{2,1}|$ and using the fact that $\Om$ is a bounded set and $u_k$ is uniformly bounded in $\Wp$, we get that 
		
		$$|D_{2,2}| =o(1) \mbox{ as } k\rightarrow +\infty.$$
		
		\noindent Then 
		\begin{equation}\label{eq31}
		|D_2| \leq |D_{2,1}|+|D_{2,2}| = o(1) \quad \Rightarrow \quad D_2 =o(1).
		\end{equation}
		
	\end{itemize}
	
	\noindent By \eqref{eq29}, \eqref{eq30}, \eqref{eq31} we obtain $$o(1)=D_2 < D_1 + D_2 = D \leq o(1),$$
	that is 
	\begin{equation}\label{=o(1)Tm+}
	D= \int_{\Om} \Big(A(|\nabla u_k|)\nabla u_k-A(|\nabla u|)\nabla u, \nabla (T_m(u_k)-T_m(u))^+\Big) \,dx = o(1).
	\end{equation}
	
	\noindent Since $(T_m(u_k)-T_m(u))^+\rightharpoonup 0$ in $\Wp$, from \eqref{=o(1)Tm+} we get 
	\begin{equation}\label{eq32}
	\int_{\Om} A(|\nabla u_k|)(\nabla u_k, \nabla (T_m(u_k)-T_m(u))^+) \,dx = o(1) \mbox{ as } k\rightarrow +\infty. 
	\end{equation}
	\\
	\noindent Since $T_m(u_k)=u_k$ on $\{0<u_k \leq m\}$, we estimate the left-hand side of \eqref{eq32} as 
	
	\begin{equation}\label{eq33}
	\begin{split}
	&o(1) = \int_{\Om} A(|\nabla u_k|)(\nabla u_k, \nabla (T_m(u_k)-T_m(u))^+) \,dx\\
	&= \int_{\Om \cap \{u_k > m\}} A(|\nabla u_k|)(\nabla u_k, \nabla (T_m(u_k)-T_m(u))^+)\,dx \\
	& + \int_{\Om \cap \{u_k \leq  m\}} A(|\nabla T_{m}(u)|)(\nabla T_{m}(u), \nabla (T_m(u_k)-T_m(u))^+)\,dx                      \\
	& + \int_{\Om \cap \{u_k \leq m\}}  \Big(A(|\nabla T_m(u_k)|)\nabla T_m(u_k)-A(|\nabla T_m(u)|)\nabla T_m(u), \nabla (T_m(u_k)-T_m(u))^+\Big) \,dx.
	\end{split}
	\end{equation}
	
	\noindent Let us denote 
	\begin{equation}\label{Chim}
	\chi_m:= \chi_{\Om \cap \{u_k> m\}}
	\end{equation}
	
	\noindent and consider the first term of the right-hand side of \eqref{eq33}. Using \eqref{gradTm+}, H\"{o}lder inequality \eqref{tA<K} and \eqref{D10}, we get
	
	\begin{equation}\label{eq34}
	\begin{split}
	&\Bigg|\int_{\Om} ( A(|\nabla u_k|)\nabla u_k, \chi_m \nabla T_m(u) )\,dx \Bigg| \\
	&\leq 
	\left(\int_{\Om} |A(|\nabla u_k|)\nabla u_k|^{\frac{p}{p-1}}\,dx\right)^{\frac{p-1}{p}}\left(\int_{\Om}|\chi_m\nabla T_m(u)|^p\,dx\right)^{\frac{1}{p}} \\
	&\leq \left( C_K^{\frac{p}{p-1}}|\Om \cap \{|\nabla u_k| < K\}| + c\int_{\Om\cap \{|\nabla u_k| \geq K\}} |\nabla u_k|^p\,dx\right)^{\frac{p-1}{p}} \norm{\chi_m\nabla T_m(u)}_p \\
	&\leq c(|\Om|+\norm{u_k}_{\Wp}^{p})^{\frac{p}{p-1}}\norm{\chi_m\nabla T_m(u)}_p.
	\end{split}
	\end{equation}
	
	\noindent Since $ \lim\limits_{k \rightarrow +\infty} \chi_m\nabla T_m(u) =0$ a.e. in $\Omega$, using dominated convergence theorem, we get 
	$$ \norm{\chi_m\nabla T_m(u)}_p \rightarrow 0 \mbox{ as } k\rightarrow +\infty$$
	and since $u_k$ is uniformly bounded in $\Wp$ on $k$ by \eqref{u_kunifb}, we obtain 
	$$ \Bigg|\int_{\Om} ( A(|\nabla u_k|)\nabla u_k, \chi_m \nabla T_m(u) )\,dx \Bigg| = o(1) \ \mbox{ as } k\rightarrow +\infty$$
	
	\noindent and in particular 
	\begin{equation}\label{1}
	\int_{\Om} ( A(|\nabla u_k|)\nabla u_k, \chi_m \nabla T_m(u) )\,dx  = o(1)\qquad \qquad \  \mbox{ as } k\rightarrow +\infty.
	\end{equation}

	\noindent For the second term of the right-hand side of \eqref{eq33}, we rewrite it as 
	\begin{equation}\label{eq35}
	\begin{split}
	&\int_{\Om \cap \{u_k \leq  m\}} A(|\nabla T_{m}(u)|)(\nabla T_{m}(u), \nabla (T_m(u_k)-T_m(u))^+)\,dx \\
	& = \int_{\Om} A(|\nabla T_{m}(u)|)(\nabla T_{m}(u), \nabla (T_m(u_k)-T_m(u))^+)\,dx \\
	& \ \ \  -\int_{\Om \cap \{u_k >  m\}} A(|\nabla T_{m}(u)|)(\nabla T_{m}(u), \nabla (T_m(u_k)-T_m(u))^+)\,dx.
	\end{split} 
	\end{equation}
	
	\noindent Since $(T_m(u_k)-T_m(u))^+\rightharpoonup 0$ in $\Wp$, the first term of the right-hand side goes to zero when $k \rightarrow +\infty$. Instead, for the second term, arguing as in \eqref{eq34}, we obtain that for $k\rightarrow +\infty$
	\begin{equation}
	\begin{split}
	&\Bigg|\int_{\Om \cap \{u_k > m\}} A(|\nabla T_{m}(u)|)(\nabla T_{m}(u), \nabla (T_m(u_k)-T_m(u))^+)\,dx \Bigg| \\
	& \leq c(|\Om|+\norm{T_m(u)}_{\Wp}^{p})^{\frac{p}{p-1}}\norm{\chi_m\nabla T_m(u)}_p = o(1) \quad \mbox{ as } k\rightarrow +\infty.
	\end{split}
	\end{equation}
	
	\noindent Then
	$$ \int_{\Om \cap \{u_k > m\}} A(|\nabla T_{m}(u)|)(\nabla T_{m}(u), \nabla (T_m(u_k)-T_m(u))^+)\,dx = o(1)$$
	
	\noindent as $k\rightarrow +\infty$ and equation \eqref{eq35} becomes
	\begin{equation}\label{2}
	\int_{\Om \cap \{u_k \leq  m\}} A(|\nabla T_{m}(u)|)(\nabla T_{m}(u), \nabla (T_m(u_k)-T_m(u))^+)\,dx =o(1).
	\end{equation}

	\noindent 
	Now, we have to estimate the last term in the right-hand side of \eqref{eq33}. We write

	\begin{equation}\label{eq36}
	\begin{split}
	&\int_{\Om \cap \{u_k \leq  m\}}  \Big(A(|\nabla T_m(u_k)|)\nabla T_m(u_k)-A(|\nabla T_m(u)|)\nabla T_m(u), \nabla (T_m(u_k)-T_m(u))^+\Big) \,dx \\
	& = \int_{\Om}  \Big(A(|\nabla T_m(u_k)|)\nabla T_m(u_k)-A(|\nabla T_m(u)|)\nabla T_m(u), \nabla (T_m(u_k)-T_m(u))^+\Big) \,dx\\
	&-\int_{\Om \cap \{u_k >  m\}}  \Big(A(|\nabla T_m(u_k)|)\nabla T_m(u_k)-A(|\nabla T_m(u)|)\nabla T_m(u), \nabla (T_m(u_k)-T_m(u))^+\Big) \,dx
	\end{split} 
	\end{equation}
	
	\noindent By \eqref{gradTm+}, the second term of the right-hand side of \eqref{eq36} becomes
	
	\begin{equation}\label{eq200}
	\begin{split}
	&\int_{\Om \cap \{u_k >  m\}} A(|\nabla T_m(u)|)|\nabla T_m(u)|^2\,dx
	\end{split}
	\end{equation}	
	
	\noindent We know that
	$$\int_{\Om \cap \{u_k >  m\}} |\nabla T_m(u)|^p\,dx = o(1) \mbox{ for } k\rightarrow +\infty,$$
	
   \noindent while using \eqref{tA<K} and \eqref{D10} we get
   
   \begin{equation}
   \begin{split}
   &\int_{\Om \cap \{u_k >  m\}} (A(|\nabla T_m(u)|)|\nabla T_m(u)|)^{\frac{p}{p-1}}\,dx \\
   & \leq C_K^{\frac{p}{p-1}} |\Om|+ c_1^{\frac{p}{p-1}} \int_{\Om \cap \{u_k >m\} \cap \{|\nabla T_m(u)| \geq  K\}}|\nabla T_m(u)|^p\,dx <+\infty
   \end{split}
   \end{equation}
	
	\noindent Then using Holder inequality in \eqref{eq200}, we obtain

	\begin{equation*}
	\begin{split}
	&\Bigg|\int_{\Om \cap \{u_k >  m\}} A(|\nabla T_m(u)|)|\nabla T_m(u)|^2\,dx \Bigg| \\
	& \leq  \left( \int_{\Om \cap \{u_k >  m\}} (A(|\nabla T_m(u)|)|\nabla T_m(u)|)^{\frac{p}{p-1}}\,dx \right)^{\frac{p-1}{p}} \left(\int_{\Om \cap \{u_k >  m\}} |\nabla T_m(u)|^p\,dx \right)^{\frac{1}{p}}      = o(1)
	\end{split}
	\end{equation*}	
	
 \noindent as  $k\rightarrow +\infty.$

	\noindent Therefore 
	\begin{equation}\label{eq37}
	\int_{\Om \cap \{u_k >  m\}} A(|\nabla T_m(u)|)|\nabla T_m(u)|^2\,dx = o(1) \quad  \mbox{ as } k\rightarrow +\infty.
	\end{equation}
	
	\noindent Using \eqref{eq37}, equation \eqref{eq36} becomes
	
	\begin{equation}\label{3}
	\begin{split}
	&\int_{\Om \cap \{u_k \leq  m\}}\Big(A(|\nabla T_m(u_k)|)\nabla T_m(u_k)-A(|\nabla T_m(u)|)\nabla T_m(u), \nabla (T_m(u_k)-T_m(u))^+\Big)\,dx \\
	& = \int_{\Om}\Big(A(|\nabla T_m(u_k)|)\nabla T_m(u_k)-A(|\nabla T_m(u)|)\nabla T_m(u), \nabla (T_m(u_k)-T_m(u))^+\Big)\,dx  +o(1)
	\end{split}
	\end{equation}
	
	\noindent Then, using  \eqref{1}, \eqref{2}, \eqref{3}, the equation \eqref{eq33} becomes
	\begin{equation*}
	\begin{split}
	&o(1)= \int_{\Om} A(|\nabla u_k|)(\nabla u_k, \nabla (T_m(u_k)-T_m(u))^+) \,dx \\
	&= \int_{\Om} \Big(A(|\nabla T_m(u_k)|)\nabla T_m(u_k)-A(|\nabla T_m(u)|)\nabla T_m(u), \nabla (T_m(u_k)-T_m(u))^+\Big) \,dx +o(1) \\
	& \mbox{ as } k \rightarrow +\infty.
	\end{split}
	\end{equation*}
	\\
	\noindent Using equations \eqref{D2} and \eqref{Dp2},  	
	\begin{equation}
	\begin{split}\displaystyle
	o(1) &=  \int_{\Om}  \Big(A(|\nabla T_m(u_k)|)\nabla T_m(u_k)-A(|\nabla T_m(u)|)\nabla T_m(u), \nabla (T_m(u_k)-T_m(u))^+\Big) \,dx \\ &\geq \begin{cases}\displaystyle
	c_2\int_{\Om} \frac{|\nabla(T_m(u_k)-T_m(u))^+|^2}{\left(|\nabla T_m(u_k)|+|\nabla T_m(u)|\right)^{2-p}}\,dx \quad & \mbox{ if } 1<p < 2,\\\displaystyle
	c_2\int_{\Om} |\nabla (T_m(u_k)-T_m(u))^+|^p\,dx & \mbox{ if } p\geq 2. 
	\end{cases}
	\end{split}
	\end{equation}
	\noindent Arguing as we have done from \eqref{eq17} to \eqref{norma2+}, we obtain
	\begin{equation}\label{Tm+}
	\norm{(T_m(u_k)-T_m(u))^+}_{\Wp} \rightarrow 0 \mbox{ as } k\rightarrow +\infty.
	\end{equation}	
	\vskip 4pt	
	\noindent For the study of the asymptotic behaviour of $\norm{(T_m(u_k)-T_m(u))^-}_{\Wp}$, we use $e^{-\hat{C}T_m(u_k)}(T_m(u_k)-T_m(u))^-$ as a test function in weak formulation of \eqref{P1_k} obtaining
	\begin{equation}\label{eq38}
	\begin{split}
	&\int_{\Om} e^{-\hat{C}T_m(u_k)}(A(|\nabla u_k|)\nabla u_k , \nabla (T_m(u_k)-T_m(u))^-)\,dx \\
	& \ \ \ - \int_{\Om} \hat{C}e^{-\hat{C}T_m(u_k)} \Big(A(|\nabla u_k|)\nabla u_k , \nabla T_m(u_k)\Big) (T_m(u_k)-T_m(u))^-  \,dx \\
	& \ \ \ + \int_{\Om}B(|\nabla u_k|) e^{-\hat{C}T_m(u_k)}(T_m(u_k)-T_m(u))^-  \,dx \\
	&= \int_{\Om}  \left(\vartheta T_k\left(\frac{u_{k}^q}{|x|^p}\right) +T_k(f)      \right )e^{-\hat{C}T_m(u_k)}(T_m(u_k)-T_m(u))^- \,dx 
	\end{split}
	\end{equation}
	
	\noindent As above, since $(T_m(u_k)-T_m(u))^- \rightharpoonup 0$ in $\Wp$ and $(T_m(u_k)-T_m(u))^-\rightarrow 0$ a.e. in $\Omega$, by dominated convergence theorem, the right-hand side of \eqref{eq38} tends to zero when $k$ goes to infinity. Moreover, splitting the domains of the integrals in the left-hand side of \eqref{eq38}, we get \\
	\begin{equation}\label{eq39}
	\begin{split}
	&\int_{\Om} e^{-\hat{C}T_m(u_k)}(A(|\nabla u_k|)\nabla u_k , \nabla (T_m(u_k)-T_m(u))^-)\,dx \\
	& \ \ \ - \int_{\Om \cap \{u_k > m\}} \hat{C}e^{-\hat{C}T_m(u_k)} \Big(A(|\nabla u_k|)\nabla u_k , \nabla T_m(u_k)\Big) (T_m(u_k)-T_m(u))^-  \,dx \\
	& \ \ \ - \int_{\Om \cap \{u_k \leq m\}} \hat{C}e^{-\hat{C}T_m(u_k)} \Big(A(|\nabla u_k|)\nabla u_k , \nabla T_m(u_k)\Big) (T_m(u_k)-T_m(u))^-  \,dx\\
	&\ \ \ + \int_{\Om \cap \{u_k > m\}}B(|\nabla u_k|) e^{-\hat{C}T_m(u_k)}(T_m(u_k)-T_m(u))^-  \,dx \\
	&\ \ \ + \int_{\Om \cap \{u_k \leq m\}}B(|\nabla u_k|) e^{-\hat{C}T_m(u_k)}(T_m(u_k)-T_m(u))^-  \,dx = o(1) \ \mbox{ as } k\rightarrow +\infty.
	\end{split}
	\end{equation}
	\\
	\noindent We have that
	\begin{equation}\label{gradTm-}
	\qquad \quad  (T_m(u_k)-T_m(u))^-=
	\begin{cases}
	(u_k - T_m(u))^-  &\mbox{ if } 0< u_k \leq m,\\
	0  & \mbox{ if } u_k > m, 
	\end{cases}
	\end{equation}
	$$\nabla (T_m(u_k)-T_m(u))^-=\begin{cases}
	\nabla(u_k - u)^-  &\mbox{ if } u_k \leq m, u\leq m,\\
	\nabla u_k  &\mbox{ if } u_k \leq m, u>m,\\
	0  & \mbox{ if } u_k > m.
	\end{cases}$$

	\noindent Recalling the definition of $\chi_m$ given in $\eqref{Chim}$, one has that $(T_m(u_k)-T_m(u))^-\chi_m=0$  and $T_m(u_k)=u_k$ on $\{0<u_k \leq m\}$, hence \eqref{eq39} becomes 
	
	\begin{equation}\label{eq40}
	\begin{split}
	&\int_{\Om} e^{-\hat{C}T_m(u_k)}(A(|\nabla u_k|)\nabla u_k , \nabla (T_m(u_k)-T_m(u))^-)\,dx \\
	&+ \int_{\Om \cap \{u_k \leq m\}} B(|\nabla T_m(u_k)|) e^{-\hat{C}T_m(u_k)}(T_m(u_k)-T_m(u))^- \,dx\\
	& - \int_{\Om \cap \{u_k \leq m\}}\hat{C}A(|\nabla T_m(u_k)|)|\nabla T_m(u_k)|^2 e^{-\hat{C}T_m(u_k)}(T_m(u_k)-T_m(u))^- \,dx\\
	& = o(1) \quad \mbox{ as } k\rightarrow +\infty.
	\end{split}
	\end{equation}

	\noindent Moreover, using \eqref{stimaBalto}, we have that
	
	\begin{equation}\label{eq201}
	\begin{split}
	&\int_{\Om} e^{-\hat{C}T_m(u_k)}\left(A(|\nabla u_k|)\nabla u_k , \nabla (T_m(u_k)-T_m(u))^-\right )\,dx \leq  o(1) \mbox{ as } k\rightarrow +\infty.
	\end{split}
	\end{equation}
	
	\noindent Recalling \eqref{gradTm-}, we define $A_1:= \{x\in \Om: u_k\leq m, u\leq m\}$ and $A_2:= \{x\in \Om: u_k\leq m, u>m\}$ and we split the integral in \eqref{eq201} as follows:
	\begin{equation}\label{*}
	\begin{split}
	&\int_{\Om} e^{-\hat{C}T_m(u_k)}\left(A(|\nabla u_k|)\nabla u_k , \nabla (T_m(u_k)-T_m(u))^-\right )\,dx \\
	&=  \int_{\Om \cap {A_1}} e^{-\hat{C}T_m(u_k)} \left(A(|\nabla u_k|)\nabla u_k , \nabla (u_k-u)^-\right ) \,dx \\
	& \ \ \ + \int_{\Om \cap A_2}e^{-\hat{C}T_m(u_k)} A(|\nabla u_k|)|\nabla u_k|^2 \,dx \leq o(1).
	\end{split}
	\end{equation}

 \noindent Since $(T_m(u_k)-T_m(u))^- \rightharpoonup 0$ in $\Wp$ for $k \rightarrow +\infty$, 
\begin{equation}\label{eq106.}
\begin{split}
&\int_{\Om} \left(A(|\nabla u_k|)\nabla u_k , \nabla (T_m(u_k)-T_m(u))^-\right )\,dx  \\
&= \int_{\Om} \left(A(|\nabla u_k|)\nabla u_k- A(|\nabla T_m(u)|)\nabla T_m(u), \nabla (T_m(u_k)-T_m(u))^-\right ) \,dx \\
&\ \ \ + \int_{\Om}\left(A(|\nabla T_m(u)|)\nabla T_m(u), \nabla (T_m(u_k)-T_m(u))^-\right ) \,dx \\ 
&= \int_{\Om} \left(A(|\nabla u_k|)\nabla u_k- A(|\nabla T_m(u)|)\nabla T_m(u), \nabla (T_m(u_k)-T_m(u))^-\right ) \,dx + o(1).\\
& = \int_{\Om \cap {A_1}} \left(A(|\nabla u_k|)\nabla u_k- A(|\nabla u|)\nabla u , \nabla (u_k-u)^-\right ) \,dx \\
& \ \ \ + \int_{\Om \cap A_2} A(|\nabla u_k|)|\nabla u_k|^2 \,dx + o(1)
\end{split}
\end{equation}

\noindent Moreover, one has that $e^{\hat{C}T_m(u_k)} \leq e^{m\hat{C}}=: \gamma_m $ uniformly on $k$.  Then, using \eqref{D2} and \eqref{*}, we estimate \eqref{eq106.} as  

\begin{equation}\label{eq107.}
\begin{split}
&\int_{\Om} \left(A(|\nabla u_k|)\nabla u_k , \nabla (T_m(u_k)-T_m(u))^-\right )\,dx   \\
&= \int_{\Om \cap {A_1}} e^{\hat{C}T_m(u_k)}e^{-\hat{C}T_m(u_k)}\left(A(|\nabla u_k|)\nabla u_k- A(|\nabla u|)\nabla u , \nabla (u_k-u)^-\right ) \,dx \\
& \ \ \ + \int_{\Om \cap A_2}  e^{\hat{C}T_m(u_k)}e^{-\hat{C}T_m(u_k)} A(|\nabla u_k|)|\nabla u_k|^2 \,dx + o(1)\\
&\leq \gamma_m \int_{\Om \cap {A_1}} e^{-\hat{C}T_m(u_k)}\left(A(|\nabla u_k|)\nabla u_k- A(|\nabla u|)\nabla u , \nabla (u_k-u)^-\right ) \,dx \\
& \ \ \ + \gamma_m \int_{\Om \cap A_2} e^{-\hat{C}T_m(u_k)} A(|\nabla u_k|)|\nabla u_k|^2 \,dx + o(1)\\
&\leq \gamma_m \Bigg[ \int_{\Om \cap {A_1}} e^{-\hat{C}T_m(u_k)}\left(A(|\nabla u_k|)\nabla u_k, \nabla (u_k-u)^-\right ) \,dx \\
& \ \ \ + \int_{\Om \cap A_2} e^{-\hat{C}T_m(u_k)} A(|\nabla u_k|)|\nabla u_k|^2 \,dx \Bigg]\\
& \ \ \ - \gamma_m \int_{\Om \cap A_1} e^{-\hat{C}T_m(u_k)}\left(A(|\nabla u|)\nabla u , \nabla (u_k-u)^-\right) \,dx + o(1)\\
& \leq - \gamma_m \int_{\Om \cap A_1} e^{-\hat{C}T_m(u_k)}\left(A(|\nabla u|)\nabla u , \nabla (u_k-u)^-\right) \,dx + (\gamma_m+1)o(1).
\end{split}
\end{equation}

\noindent Now, we estimate the right hand-side of \eqref{eq107.} in this way: 

\begin{equation}\label{eq900.}
\begin{split}
&\int_{\Om \cap A_1} e^{-\hat{C}T_m(u_k)}\left(A(|\nabla u|)\nabla u , \nabla (u_k-u)^-\right) \,dx\\ 
&=\int_{\Om} e^{-\hat{C}T_m(u_k)}\left(A(|\nabla T_m(u)|)\nabla T_m(u) , \nabla (T_m(u_k)-T_m(u))^-\right) \,dx\\
&=  \int_{\Om} e^{-\hat{C}(T_m(u_k)-T_m(u))^-} e^{-\hat{C}T_m(u)}\left(A(|\nabla T_m(u)|)\nabla T_m(u) , \nabla (T_m(u_k)-T_m(u))^-\right) \,dx
\\
&= -\frac{1}{\hat{C}}  \int_{\Om}e^{-\hat{C}T_m(u)}A(|\nabla T_m(u)|)\big(\nabla T_m(u), \nabla( e^{-\hat{C}(T_m(u_k)-T_m(u))^-}-1)\big) \,dx.
\end{split}
\end{equation}

\noindent Since $\{ e^{-\hat{C}(T_m(u_k)-T_m(u))^-}-1\}_{k\in \mathbb{N}}$ is uniformly bounded in $W_0^{1,p}(\Om)$, up to a subsequence there exists $g\in W_0^{1,p}(\Om)$ such that

\begin{equation}\label{convg.}
e^{-\hat{C}(T_m(u_k)-T_m(u))^-}-1 \rightharpoonup g \mbox{ as } n\rightarrow +\infty.
\end{equation}

\noindent In particular we note that $g \equiv 0$ since $T_m(u_k) \rightarrow T_m(u)$ a.e. in $\Om$ as $k\rightarrow +\infty.$

\noindent Hence, by \eqref{convg.}, for $k\rightarrow +\infty$

\begin{equation}\label{eq108.}
\int_{\Om}e^{-\hat{C}T_m(u)}A(|\nabla T_m(u)|)(\nabla T_m(u), \nabla (e^{-\hat{C}(T_m(u_k)-T_m(u))^-}-1)) \,dx \rightarrow 0. 
\end{equation}

\noindent Then, using \eqref{eq900.} and \eqref{eq108.} in \eqref{eq107.}, one has that 

	\vskip 20pt
	\begin{equation}\label{3<o(1)}
	\begin{split}
	&\int_{\Om} \left(A(|\nabla u_k|)\nabla u_k , \nabla (T_m(u_k)-T_m(u))^-\right )\,dx \leq  o(1) \mbox{ as } k\rightarrow +\infty.
	\end{split}
	\end{equation}
	
	\noindent Arguing in the similar way as we have done from \eqref{eq27} to \eqref{eq32}, we get
	\begin{equation}\label{3=o(1)}
	\begin{split}
	&\int_{\Om} \left(A(|\nabla u_k|)\nabla u_k , \nabla (T_m(u_k)-T_m(u))^-\right )\,dx =  o(1) \mbox{ as } k\rightarrow +\infty. 
	\end{split}
	\end{equation}

	\noindent In order to obtain the desired result, we proceed writing the left-hand side of \eqref{3=o(1)} as
	\begin{equation}\label{eq41}
	\begin{split}
	&o(1)  = \displaystyle \int_{\Om} A(|\nabla u_k|)\left(\nabla u_k , \nabla (T_m(u_k)-T_m(u))^-\right )\,dx\\
	&= \int_{\Om \cap \{u_k > m\}} A(|\nabla u_k|)(\nabla u_k, \nabla (T_m(u_k)-T_m(u))^-)\,dx \\
	& + \int_{\Om \cap \{u_k \leq  m\}} A(|\nabla T_{m}(u)|)(\nabla T_{m}(u), \nabla (T_m(u_k)-T_m(u))^-)\,dx                      \\
	&+ \int_{\Om \cap \{u_k \leq m\}}  \Big(A(|\nabla T_m(u_k)|)\nabla T_m(u_k)-A(|\nabla T_m(u)|)\nabla T_m(u), \nabla (T_m(u_k)-T_m(u))^-\Big) \,dx.
	\end{split}
	\end{equation}
	
	\noindent The first term of the right-hand side of \eqref{eq41}, by \eqref{gradTm-}, is zero since $(T_m(u_k)-T_m(u))^-\chi_m=0$.
	
	\noindent The second term of the right-hand side of \eqref{eq41} can be rewritten as 
	\begin{equation}\label{eq42}
	\begin{split}
	\displaystyle
	&\int_{\Om \cap \{u_k \leq  m\}} A(|\nabla T_{m}(u)|)(\nabla T_{m}(u), \nabla (T_m(u_k)-T_m(u))^-)\,dx  \\
	& = \int_{\Om} A(|\nabla T_{m}(u)|)(\nabla T_{m}(u), \nabla (T_m(u_k)-T_m(u))^-)\,dx \\
	& \ \ \ -\int_{\Om \cap \{u_k >  m\}} A(|\nabla T_{m}(u)|)(\nabla T_{m}(u), \nabla (T_m(u_k)-T_m(u))^-)\,dx.
	\end{split} 
	\end{equation}
	
	\noindent Since $(T_m(u_k)-T_m(u))^-\rightharpoonup 0$ in $\Wp$, the first term of the right-hand side goes to zero when $k \rightarrow +\infty$. While, the second term, is zero since $(T_m(u_k)-T_m(u))^-\chi_m=0$.\\
	
	\noindent Then, the equation \eqref{eq41} becomes
	\begin{equation}\label{eq43}
	\begin{split}
	o(1) &= \int_{\Om \cap \{u_k \leq m\}}  \Big(A(|\nabla T_m(u_k)|)\nabla T_m(u_k)-A(|\nabla T_m(u)|)\nabla T_m(u), \\
	& \ \ \ ...\nabla (T_m(u_k)-T_m(u))^-\Big) \,dx.
	\end{split}
	\end{equation}

	\noindent Now, we estimate the the right-hand side of \eqref{eq43}.\\ We write
	\begin{equation}\label{eq44}
	\begin{split}
	&\int_{\Om \cap \{u_k \leq  m\}}  \Big(A(|\nabla T_m(u_k)|)\nabla T_m(u_k)-A(|\nabla T_m(u)|)\nabla T_m(u), \nabla (T_m(u_k)-T_m(u))^-\Big) \,dx \\
	& = \int_{\Om}  \Big(A(|\nabla T_m(u_k)|)\nabla T_m(u_k)-A(|\nabla T_m(u)|)\nabla T_m(u), \nabla (T_m(u_k)-T_m(u))^-\Big) \,dx\\
	&-\int_{\Om \cap \{u_k >  m\}}  \Big(A(|\nabla T_m(u_k)|)\nabla T_m(u_k)-A(|\nabla T_m(u)|)\nabla T_m(u), \nabla (T_m(u_k)-T_m(u))^-\Big) \,dx.
	\end{split} 
	\end{equation}
	
	\noindent By \eqref{gradTm-}, the second term of the right-hand side is zero then equation \eqref{eq43} becomes
	\begin{equation}
	o(1) = \int_{\Om}  \Big(A(|\nabla T_m(u_k)|)\nabla T_m(u_k)-A(|\nabla T_m(u)|)\nabla T_m(u), \nabla (T_m(u_k)-T_m(u))^-\Big) \,dx.
	\end{equation}
	
	\noindent From equation \eqref{D2} and \eqref{Dp2},  	
	\begin{equation}
	\begin{split}
	o(1) &= \int_{\Om}  \Big(A(|\nabla T_m(u_k)|)\nabla T_m(u_k)-A(|\nabla T_m(u)|)\nabla T_m(u), \nabla (T_m(u_k)-T_m(u))^-\Big) \,dx \\ &\geq \begin{cases}
	\displaystyle
	c_2\int_{\Om} \frac{|\nabla(T_m(u_k)-T_m(u))^-|^2}{\left(|\nabla T_m(u_k)|+|\nabla T_m(u)|\right)^{2-p}}\,dx \quad & \mbox{ if } 1<p < 2,\\
	\displaystyle c_2\int_{\Om} |\nabla (T_m(u_k)-T_m(u))^-|^p\,dx & \mbox{ if } p\geq 2. 
	\end{cases}
	\end{split}
	\end{equation}
	
	\noindent Arguing as we have done from \eqref{eq17} to \eqref{norma2+}, we obtain
	\begin{equation}\label{Tm-}
	\norm{(T_m(u_k)-T_m(u))^-}_{\Wp} \rightarrow 0 \mbox{ as } k\rightarrow +\infty.
	\end{equation}

	\noindent Summarizing up, from \eqref{normTm} we get
	\begin{equation}\label{eq300}
	\norm{T_m(u_k)-T_m(u)}_{\Wp} \rightarrow 0 \mbox{ as } k\rightarrow +\infty.
	\end{equation}	
\end{proof}

\noindent Using Lemma \ref{lemmaconvTm}, up to a subsequence, we get

\begin{equation}\label{eq202}
|\nabla T_m(u_k)|^p \rightarrow |\nabla T_m(u)|^p \mbox{ a.e. in  } \Om, \mbox{ for } k\rightarrow +\infty.
\end{equation}

\noindent For almost everywhere $x$ fixed in $\Om$, we choose $\eta >> |u(x)|$. Since $u_k \rightarrow u$ a.e. in $\Om$ for $k\rightarrow +\infty$, we get that $|u_k(x)|<< \eta$ definitely. \\ Choosing $m=\eta$ in \eqref{eq202}, one has
\begin{equation*}
|\nabla u_k|^p \rightarrow |\nabla u|^p \mbox{ a.e. in } \Omega \mbox{ for } k\rightarrow +\infty.
\end{equation*}

%
%

\noindent In order to use Vitali Theorem we need to prove the uniform-integrability of $|\nabla u_k|^p$. Let $E \subset \Omega$ be a measurable set, then 
\begin{equation}\label{eq45}
\begin{split}
\int_{E}|\nabla u_k|^p\,dx &= \int_{E \cap \{u_k > m\}} |\nabla u_k|^p\,dx + \int_{E \cap \{u_k \leq m\}} |\nabla u_k|^p\,dx \\
& =  \int_{E \cap \{u_k > m\}} |\nabla u_k|^p\,dx +  \int_{E} |\nabla T_m(u_k)|^p\,dx.
\end{split}
\end{equation}

\noindent For fixed $\varepsilon >0$, using Lemma \ref{lim0} in the first term of \eqref{eq45}, there exists $\bar{m}$ that does not depend on $k$, such that 
\begin{equation*}
\int_{E \cap \{u_k > m\}} |\nabla u_k|^p\,dx \leq \int_{\Omega \cap \{u_k > m\}} |\nabla u_k|^p\,dx \leq \frac{\varepsilon}{2}.
\end{equation*}

\noindent By $\eqref{eq300}$ and up to a subsequence, there exists $v(x)\in L^1(\Omega)$ such that $|\nabla T_{\bar{m}}(u_k)|^p \leq v(x)$ a.e. in $\Omega$. So for $|E|$ small we get 
\begin{equation*}
\int_{E} |\nabla T_{\bar{m}}(u_k)|^p\,dx \leq \frac{\varepsilon}{2}.
\end{equation*}

\noindent Therefore, \begin{equation*}
\int_{E}|\nabla u_k|^p\,dx \leq \varepsilon.
\end{equation*}
is uniformly integrable on $k$.

Collecting the previous result, the left-hand side of \eqref{eq45} is uniformly integrable on $k$.\\

\noindent Then Vitali Theorem implies that
\begin{equation}\label{ukconvl1}
|\nabla u_k|^p \rightarrow |\nabla u|^p \mbox{ in } L^1(\Omega),
\end{equation}
that is equivalent to 
\begin{equation}\label{Cnorme}
\norm{u_k}_{\Wp} \rightarrow \norm{u}_{\Wp}.
\end{equation}

\vskip 20pt
\noindent \textit{Step 4:} In order to obtain a solution to the problem \eqref{P}, we note that by \cite[Theorem 1]{BrezisLieb}
%
%
%
$$\lim\limits_{k\rightarrow +\infty}\norm{u_k-u}_{\Wp} =0.$$

\noindent Moreover we observe that $B(|\nabla u_k|) \rightarrow B(|\nabla u|)$ in $L^1(\Omega)$. \\ Since B is a continuous function 
$$B(|\nabla u_k|) \rightarrow B(|\nabla u|) \mbox{ a.e. in } \Omega.$$

\noindent From \eqref{ukconvl1}, up to a subsequence, there exists $v(x)\in L^1(\Omega)$ such that $|\nabla u_k|^p \leq v(x)$ a.e. in $\Omega$.

\noindent Then, if $E \subset \Omega$ is a measurable set, using \eqref{stimaBalto}, \eqref{tA<K} and \eqref{D10}, we have that
\begin{equation*}
\begin{split}
&\int_{E} B(|\nabla u_k|)\,dx  \leq \hat{C}\int_{E} |A(|\nabla u_k|)||\nabla u_k|^2 \,dx \\
 &\leq KC_K\hat{C}|E\cap \{|\nabla u_k| < K\}| + c_1\hat{C} \int_{E\cap \{|\nabla u_k| \geq K\}} |\nabla u_k|^p\,dx \\
 & \leq KC_K|E| + c_1\hat{C} \int_E v(x)\,dx
\end{split}
\end{equation*}

\noindent and, since $v(x)\in L^1(\Omega)$, we get that $B(|\nabla u_k|)$ is uniformly integrable.\\

\noindent Using Vitali Theorem we have
$$ B(|\nabla u_k|) \rightarrow B(|\nabla u|)  \mbox{ in } L^1(\Omega).$$

\noindent Then, collecting all the previous results, by dominate convergence, passing to the limit for $k\rightarrow +\infty$ in 
\begin{equation*}
\begin{split}
\int_{\Omega} A(|\nabla u_k|)(\nabla u_k, \nabla \varphi )\,dx +\int_{\Om} B(|\nabla u_k|)\varphi\,dx  = \vartheta \int_{\Om} T_k&\left(\frac{u_k^q}{|x|^p}\right)\varphi \,dx+ \int_{\Om} T_k(f)\varphi\,dx \\& \forall \varphi \in W_0^{1,p}(\Om) \cap L^{\infty}(\Omega). 
\end{split}
\end{equation*}

\noindent we obtain
\begin{equation*}
\begin{split}
\int_{\Omega} A(|\nabla u|)(\nabla u, \nabla \varphi )\,dx +\int_{\Om} B(|\nabla u|)\varphi  \,dx= \vartheta \int_{\Om} &\frac{u^q}{|x|^p}\varphi \,dx + \int_{\Om} f \varphi\,dx \\
& \forall \varphi \in W_0^{1,p}(\Om) \cap L^{\infty}(\Omega). 
\end{split}
\end{equation*}

\end{proof}

\section{Regularity of the energy solution}
 \noindent Taking into account that our solution $u\in \Wp$, by \cite{Di,T} we get $C_{loc}^{1,\alpha}(\Omega\setminus \{0\})$ regularity for our solution; supposing that $\Omega$ is smooth, $C_{loc}^{1,\alpha}(\bar{\Omega}\setminus \{0\})$ regularity follows by \cite{Lieberman} while, from \cite{GT}, $u\in C^2(\Om \setminus (Z_u \cup \{0\}))$ follows. For this reason in what follows we will indicate by 
\begin{equation}\label{DerSeconde}\tilde{u}_{ij}(x):=
\begin{cases} \displaystyle
{u}_{x_ix_j}(x) & x\in \Om \setminus (Z_u \cup \{0\}),\\
0 & x\in Z_u
\end{cases}
\end{equation}

\noindent and $\tilde{\nabla}u_i$ stands for ''gradient'' of $(\tilde{u}_{i1}, ..., \tilde{u}_{iN})$.\\
\indent Next results give us the summability properties of the second derivative, main tool for the analysis of the qualitative properties of solutions of elliptic problems. The proofs follow the ideas introduced in \cite{ CanDeGS, DamSciu,EspST} therefore we include only some details.\\\\
\noindent We will assume that $A$ satisfies \eqref{D10}, \eqref{D2} and also 
\begin{equation}\label{hpregolarita}
-1< \inf\limits_{t> 0} \frac{tA'(t)}{A(t)}=: m_A \leq M_A:= \sup\limits_{t> 0} \frac{tA'(t)}{A(t)} <+\infty.
\end{equation}
\begin{rem}\label{teAlim}
  Recalling \cite[Proposition 4.1]{Cia14} it is already proved that
  $$
   A(1)\min\{t^{m_A},t^{M_A}\}\leq A(t)\leq A(1)\max\{t^{m_A},t^{M_A}\}.
  $$
Since $m_A>-1$ there exists $\eta\in[0,1)$ such that $m_A+\eta>0$; hence
$$
 \lim_{t\to 0}t^\eta A(t)=0
$$
and, moreover, it easy to check that $tA(t)$ is non-decreasing on $[0,+\infty)$.
\end{rem}

\noindent The following results deal with the study of regularity of solutions to the following problem
\begin{equation}\label{ProbPuntuale}
\begin{cases}
-\mbox{div}(A(|\nabla u|)\nabla u)+B(|\nabla u|)=h(x,u), &\mbox{ in } \Om,\\
u>0 & \mbox{ in } \Om,\\
u=0& \mbox{ on } \partial \Om,
\end{cases}
\end{equation}
 
\noindent where we set $ \displaystyle h(x,u):=\vartheta \frac{u^q}{|x|^p}+f(x)$ and $\displaystyle u_i:=\frac{\partial u}{\partial x_i}$.\\
\begin{thm}\label{Stima_D2_pt1}
	Assume that $\Om$ is a bounded smooth domain and $1<p<N$. Consider $u\in C^{1,\alpha}_{loc}(\bar{\Omega} \setminus \{0\})\cap C^2(\Om \setminus (Z_u \cup \{0\}))$ a solution to \eqref{ProbPuntuale},
	where $f \in W^{1,\infty}(\bar{\Om})$. We have 	
	
	\begin{equation}\label{reg1}
 \int_{E \setminus Z_u} \frac{A(|\nabla u|)|\nabla u_i|^2}{|x-y|^{\gamma}|u_i|^{\beta}} \,dx \leq \mathcal{C} \quad \forall i=1,...,N
	\end{equation}
	
	\noindent for any $E \Subset \Om \setminus \{0\}$ and uniformly for any $y \in E$, with 
	$$\mathcal{C}:= \mathcal{C}(\gamma,m_A,M_A, \beta, h,\|\nabla u\|_\infty,\rho)$$
	for $0 \leq \beta <1$ and $\gamma < (N-2)$ if $N \geq 3$ ($\gamma=0$ if $N=2$).\\
	
	\noindent Moreover, if we also assume that $f$ is nonnegative in $\Om$, we have that 
	
	\begin{equation}\label{reg1.2}
	\int_{\Om \setminus (Z_u \cup \{0\})} \frac{A(|\nabla u|)|\nabla u_i|^2}{|x-y|^{\gamma}|u_i|^{\beta}} \,dx \leq \mathcal{C} \quad \forall i=1,...,N
	\end{equation}
\end{thm}

\begin{proof}	
	\noindent \textbf{I step:} Since $u>0$ in $\Om$, we have that $h(x,u)$ is locally Lipschtz with respect to $u$, uniformly in $x$ and $h(\cdot,u)$ is locally Lipschtz. Following \cite[Lemma 2.1]{DamSciu},  we can state that $u_i$ solves the linearized problem
	\begin{equation}  \label{linequ}
	\begin{split}
	&\int_\Omega A(|\nabla u|)(\tilde{\nabla} u_{i},\nabla \psi )\,dx+
	\int_\Omega \frac{A^{\prime }\left( \left\vert \nabla u\right\vert \right) }{\left\vert \nabla u\right\vert }\left( \nabla u,\tilde{\nabla} u_{i}\right) \left( \nabla u,\nabla \psi \right) dx\\
	&+\int_\Omega \frac{B^{\prime }\left( \left\vert \nabla u\right\vert \right) }{\left\vert \nabla u\right\vert }\left( \nabla u,\tilde{\nabla} u_{i}\right) \psi \,dx-\int_\Omega [h_{i}(x,u)+h_u(x,u)u_i] \, \psi \,dx=0,
	\end{split}
	\end{equation}
	for any $\psi\in C_c^\infty(\Om \setminus (Z_u \cup \{0\}))$.\\

	\noindent For every $\varepsilon, \delta >0$, we consider the following function:
	\begin{eqnarray*}
		&&G_{\varepsilon }(t):= (2t-2\varepsilon )\chi _{\lbrack \varepsilon
			,2\varepsilon ]}(t)+t\chi _{\lbrack 2\varepsilon ,\infty )}(t)  \quad \text{for } t>0,\\
		&&T_{\varepsilon }(t):=\frac{G_{\varepsilon }(t)}{|t|^{\beta }}\ , \ \ H_{\delta}(t):= \frac{G_{\delta}(t)}{|t|^{\gamma+1}}
	\end{eqnarray*}
	
	\vskip 10pt
	\noindent Then, if we consider $x_0 \in \Om \setminus \{0\}$, $B_{2\rho}(x_0)\subset \Om \setminus \{0\}$ and $\varphi_\rho \in C_{c}^{\infty }(B_{2\rho }(x_{0}))$ such that $\varphi_\rho =1\text{ in }B_{\rho }(x_{0})\text{ and }|\nabla \varphi_\rho |\leq \frac{%
		2}{\rho }$, arguing as in \cite[Theorem 1.2]{EspST} and 
	using the test function
	\begin{equation*}
	\psi= T_{\varepsilon}(u_i)H_{\delta}(|x-y|)\varphi_{\rho}^2,
	\end{equation*}
	we get that
	
	\begin{equation}\label{stimaprec}
	\begin{split}
	\min\{1,&1+m_{A}\}\int_{B_{2\rho}(x_0)} \frac{A(|\nabla u|)|\tilde{\nabla} u_{i}|^{2}}{|x-y|^{\gamma}} \left(\frac{G'_{\varepsilon}(u_{i})}{|u_{i}|^{\beta}}-\beta 
	\frac{G_{\varepsilon}(u_{i})}{|u_{i}|^{1+\beta}}\right) \varphi_\rho^2\,dx \\
	&-3\tilde{\vartheta} \int_{B_{2\rho}(x_0)} \frac{A(|\nabla u|)|\tilde{\nabla} u_{i}|^{2}}{|x-y|^\gamma |u_{i}|^\beta} \chi_{\{|u_{i}|\geq \varepsilon\}}\varphi_\rho^2\,dx \leq C_{3,4,5,6,7} \\
	& +\limsup_{\delta\to 0}\int_{B_{2\rho}(x_0)}[h_{i}(x,u)+h_u(x,u)u_i]T_\varepsilon(u_i)H_\delta(|x-y|)\varphi_\rho^2\,dx.
	\end{split}
	\end{equation}

	\noindent Exploiting the definition of $T_{\varepsilon}(t)$ and $H_{\delta}(t)$, we can estimate the last term of the right-hand side of \eqref{stimaprec} as 
	
	\begin{equation*}
	\begin{split}
	& \left|\int_{B_{2\rho}(x_0)}[h_{i}(x,u)+h_u(x,u)u_i] \ T_\varepsilon(u_i) \ H_\delta(|x-y|)\varphi_\rho^2\,dx \right|\\
	& \leq \int_{B_{2\rho}(x_0)}\frac{|h_{i}(x,u)||u_i|^{1-\beta}}{|x-y|^{\gamma}}\varphi^2_{\rho}\,dx + \int_{B_{2\rho}(x_0)}\frac{|h_{u}(x,u)||u_i|^{2-\beta}}{|x-y|^{\gamma}}\varphi^2_{\rho}\,dx \\
	&\leq \sup\limits_{x\in B_{2\rho}(x_0)} |h_{i}(x,u)| \norm{\nabla u}_{\infty}^{1-\beta}\int_{B_{2\rho}(x_0)} \frac{1}{|x-y|^{\gamma}}\,dx \\
	& \ \ \  + \norm{h_u(x,u)}_{\infty}\norm{\nabla u}_{\infty}^{2-\beta}\int_{B_{2\rho}(x_0)} \frac{1}{|x-y|^{\gamma}}\,dx \leq C_8. 
	\end{split}
	\end{equation*}

	\noindent Passing to the limit for $\varepsilon \rightarrow 0$ and applying Fatou Lemma in \eqref{stimaprec}, we obtain
	\begin{equation*}
	\begin{split}
	\left( \min\{1,1+m_{A}\}(1-\beta)-3\tilde{\vartheta} \right)&\int_{B_{2\rho}(x_{0})\setminus Z_u}\frac{A(|\nabla
		u|)|\nabla u_{i}|^{2}}{|x-y|^{\gamma }|u_{i}|^{\beta }}\varphi_\rho ^{2}\,dx\\
	&\leq C_{3,4,5,6,7} + C_8.
	\end{split}
	\end{equation*}

	\noindent Now we fix $\tilde{\vartheta} $ sufficiently small such that%
	\begin{equation*}
	\min\{1,1+m_{A}\}(1-\beta) -3\tilde{\vartheta} >0
	\end{equation*}
	so that,
	\begin{equation*}
	\int_{B_{\rho}(x_0) \setminus Z_u}\frac{A(|\nabla u|)|\nabla u_{i}|^{2}}{|x-y|^{\gamma
		}|u_{i}|^{\beta }}\,dx\leq \int_{B_{2\rho}(x_0) \setminus Z_u}\frac{A(|\nabla
		u|)|\nabla u_{i}|^{2}}{|x-y|^{\gamma }|u_{i}|^{\beta }}\varphi_\rho ^{2}\,dx\leq \mathcal{C}_0 
	\end{equation*}
	where $\mathcal{C}_0=\mathcal{C}_0(\gamma,m_A,M_A, \beta, h,\|\nabla u\|_\infty,\rho ,x_{0})$, for all $i=1,...,N$.\\

	\vskip 10pt
	
	\noindent \textbf{II step:} 
	\noindent Let us consider $E \Subset \Om \setminus \{0\}$.
	Since $\overline{E}$ is a compact set, there exists a finite number of $B_{\rho}(x_i)$ with $x_i \in \overline{E}$ such that 
	$$E \subset \overline{E} \subset \bigcup_{i=1}^{M}B_{\rho}(x_i).$$ 
	\noindent Hence, repeating the covering arguments exploited in \cite[Theorem 4.1]{CanDeGS} and by using the estimate obtained in Step I, we get
	\begin{equation}\label{SomDerS}
	\begin{split}
	\int_{E \setminus Z_u}\frac{A(|\nabla u|)|\nabla u_{i}|^{2}}{|x-y|^{\gamma
		}|u_{i}|^{\beta }}\,dx &\leq \sum_{i=1}^{M}\int_{B_{\rho}(x_i) \setminus Z_u}\frac{A(|\nabla u|)|\nabla u_{i}|^{2}}{|x-y|^{\gamma
		}|u_{i}|^{\beta }}\,dx + c \leq  \sum_{i=1}^{M} \mathcal{C}_i + c=:\mathcal{C}.
	\end{split}
	\end{equation}
	
%
	
	\noindent \textbf{III step:} If $\Om$ is a smooth domain and $f$ is a nonnegative function in $\Om$, by H\"{o}pf Lemma we get that $Z_u \cap \partial \Om = \emptyset$. Then, we can take $\delta > 0$ such that 
	$$|\nabla u| \neq 0 \mbox { in } I_{3\delta}(\partial \Om) \mbox{ and } A(|\nabla u|)>0 \mbox{ in } I_{3\delta}(\partial \Om).$$
	
	\noindent Since $u\in C^2(I_{3\delta}(\partial \Om))$, arguing as in \cite[Theorem 4.1]{CanDeGS} and using \eqref{SomDerS} we get \eqref{reg1.2}.
\end{proof}

\begin{cor}\label{CorforLou}
	Let $u\in C^1(\Om \setminus \{0\})$ be a weak solution to \eqref{ProbPuntuale} with $f \in W^{1,\infty}(\Om)$, $1< p< \infty$. Then $A(|\nabla u|)|\nabla u| \in W^{1,2}_{loc}(\Om \setminus \{0\}, \Real^N)$. \ \\
If $\Om$ is smooth, $u \in C^1(\bar{\Om} \setminus \{0\})$, $f$ is a nonnegative function and $f \in W^{1,\infty}(\bar{\Om})$, then $A(|\nabla u|)|\nabla u| \in W_{loc}^{1,2}(\bar{\Om} \setminus \{0\}, \Real^N)$.
\end{cor}

\begin{proof}
	Since $h(x,u)$ is locally Lipschitz continuous, $u\in C^2(\Om \setminus (Z_u \cup \{0\}))$ (see \cite{GT}). Let us consider a compact set $E \subset \Om \setminus \{0\}$ and let us set 
	$$\displaystyle \phi_n \equiv G_{\frac{1}{n}}(A(|\nabla u|)u_{i}), \quad i=1, ..., N,$$
	
	\noindent where $\displaystyle G_{\frac{1}{n}}$ is defined as in Theorem \ref{Stima_D2_pt1}.
	
	\noindent Using definition of $G_{\frac{1}{n}}$, we get that $G'_{\frac{1}{n}}(A(|\nabla u|)u_{i}) \equiv 0$ on $Z_u$, then
	\begin{equation*}
	\begin{split}
	\frac{\partial \phi_n}{\partial x_j} = G'_{\frac{1}{n}}(A(|\nabla u|)u_{i}) \left[\frac{A'(|\nabla u|)}{|\nabla u|}<\nabla u, \nabla u_j> u_{i}+A(|\nabla u|)u_{ij}\right] \mbox{ on } E\setminus Z_u.
	\end{split}
	\end{equation*}
	
	\noindent We observe that $\phi_n \in W^{1,2}(E)$. 
	
	\noindent In fact, choosing $0<\eta<1$, using \eqref{hpregolarita}, Remark \ref{teAlim} and Theorem \ref{Stima_D2_pt1} we get
	
	\begin{equation*}
	\begin{split}
	&\int_E |\nabla \phi_n|^2\,dx \leq 4 \int_{E \setminus \left\{|A(|\nabla u|)u_{i}| \leq \frac{1}{n} \right\}} \sum_{j=1}^N \left[\frac{\partial }{\partial x_j}A(|\nabla u|)u_{i}\right]^2\,dx \\
	& \leq 4 \int_{E \setminus Z_u} \sum_{j=1}^N \left[ \frac{A'(|\nabla u|)}{|\nabla u|}<\nabla u, \nabla u_{j}> u_{i}+A(|\nabla u|)u_{ij}\right]^2\,dx\\
	& \leq 4 \int_{E \setminus Z_u} \sum_{j=1}^N \left[ \frac{A'(|\nabla u|)|\nabla u|}{|\nabla u|}|\nabla u_{j}| |\nabla u| +A(|\nabla u|)u_{ij}\right]^2\,dx\\
	& \leq 4 \int_{E \setminus Z_u} \sum_{j=1}^N \left[ M_A A(|\nabla u|)\norm{D^2u} +A(|\nabla u|)\norm{D^2u}\right]^2\,dx\\
	&\leq  4N \int_{E \setminus Z_u} (1+M_A)^2 A^2(|\nabla u|)\norm{D^2u}^2\,dx\\
	&= 4N(1+M_A)^2  \int_{E \setminus Z_u}A(|\nabla u|)|\nabla u|^{\eta} \frac{A(|\nabla u|)\norm{D^2u}^2}{|\nabla u|^{\eta}}\,dx\\
	&\leq  4N(1+M_A)^2 \sup\limits_{E\setminus Z_u} A(|\nabla u|)|\nabla u|^{\eta} \int_{E \setminus Z_u}\frac{A(|\nabla u|)\norm{D^2u}^2}{|\nabla u|^{\eta}}\,dx \leq K_1 \quad \forall n\in \mathbb{N}.
	\end{split}
	\end{equation*}
	
	\noindent In the same way we can prove that $\norm{\phi_n}_{2}^2 \leq K_2$ $\forall n \in \mathbb{N}$.
	\noindent Since $W^{1,2}(E)$ is a reflexive space and $\norm{\phi_n}_{W^{1,2}(E)}\leq K := \sqrt{K_1+K_2}$ \  $\forall n \in \mathbb{N}$, up to a subsequence, there exists $h\in W^{1,2}(E)$ such that
	$$\phi_n \rightharpoonup h \mbox{ in } W^{1,2}(E)$$
	
	\noindent Moreover, using Sobolev compact embedding, we get
	\begin{align*}
	&\phi_n \rightarrow h \mbox{ in } L^2(E)\\
	&\phi_n \rightarrow h \mbox{ almost everywhere in } E.
	\end{align*}
	
	\noindent Since $\phi_n\xrightarrow{n \rightarrow +\infty} A(|\nabla u|)u_{i}$ almost everywhere in $E$, we get 
	
	$$ A(|\nabla u|)u_{i} \equiv h \in W^{1,2}(E)$$
	
	\noindent and by arbitrariness of $i\in \{1,...,N\}$ we have that $A(|\nabla u|)\nabla u \in  W^{1,2}_{loc}(\Om \setminus \{0\}, \Real^N)$.
	
	\vskip 10 pt
	\noindent If moreover $\Om$ is smooth, $u \in C^1(\bar{\Om} \setminus \{0\})$, $f$ is a nonnegative function and $f \in W^{1,\infty}(\Om)$, we get that $u\in C^2(\bar{\Om} \setminus (Z_u \cup \{0\}))$ and $A(|\nabla u|)\nabla u \in  W_{loc}^{1,2}(\bar{\Om} \setminus \{0\},\Real^N)$.

\end{proof}

\begin{rem}\label{Rem|Zu|=0}
	If $u$ is a weak solution of Problem \eqref{P}
	with $f$ nonnegative, by Corollary \ref{CorforLou}, since $A(|\nabla u|)|\nabla u| \in W_{loc}^{1,2}(\bar{\Om}\setminus \{0\})$, we get that $u$ solves Problem \eqref{P} almost everywhere in $\Om$.
	In particular, since $B(0)=0$, we get that $-\mbox{div}(A(|\nabla u|)\nabla u)=0$ on $Z_u$, that is absurd since $f$ is nonnegative.\\ Then 
	$$|Z_u|=0.$$
\end{rem}

\begin{rem}\label{IstimasuOm}
	Using Remark \ref{Rem|Zu|=0}, since $f$ is a nonnegative function, \eqref{reg1.2} becomes

	\begin{equation*}
	\int_{\tilde{\Om}} \frac{A(|\nabla u|)|\nabla u_i|^2}{|x-y|^{\gamma}|u_i|^{\beta}} \,dx \leq \mathcal{C} \quad \forall i=1,...,N
	\end{equation*}

	\noindent for any $\tilde{\Om} \Subset \Om \setminus \{0\}$ and uniformly for any $y \in \tilde{\Om}$, with 
	$$\mathcal{C}:= \mathcal{C}(\gamma,m_A,M_A, \beta, h,\|\nabla u\|_\infty,\rho)$$
	for $0 \leq \beta <1$ and $\gamma < (N-2)$ if $N \geq 3$ ($\gamma=0$ if $N=2$).\\
	
	\noindent Using \eqref{D20} and the fact that $|u_i| \leq |\nabla u|$, we also obtain
	\begin{equation*}
		c_2\int_{\tOm} \frac{|\nabla u|^{p-2-\beta}|\nabla u_i|^2}{|x-y|^{\gamma}} \,dx \leq
	\int_{\tilde{\Om}} \frac{A(|\nabla u|)|\nabla u_i|^2}{|x-y|^{\gamma}|u_i|^{\beta}} \,dx \leq \mathcal{C} \quad \forall i=1,...,N.
	\end{equation*}
	
\end{rem}

\noindent Next Theorems concern the summability of $(A(|\nabla u|))^{-1}$. The first is obtained supposing that $p>2$.
\begin{thm}
	\label{Local_sum_weight}
	Let $u\in C^{1,\alpha}(\Omega\setminus \{0\})\cap C^{2}(\Omega\setminus (Z_u \cup \{0\}))$ be a solution to \eqref{ProbPuntuale} with $f\in W^{1,\infty}(\Omega)$ and
	$f(x)>0$ in $B_{2 \rho}(x_0) \subset \Omega \setminus \{0\}$ (where $B_{2 \rho}(x_0)$ is defined in the proof of Theorem \ref{Stima_D2_pt1}) . Let us suppose that $p>2$. Then 
	\begin{equation}
	\label{Eq_2}
	\int_{B_{\rho}(x_0)} \frac{1}{(A(|\nabla u|))^{\alpha r}} \frac{1}{|x-y|^\gamma} dx\le \mathcal{C}
	\end{equation}
	for any $y\in B_{\rho}(x_0)$, with $\displaystyle \alpha:=\frac{p-1}{p-2}$, $r\in (0,1)$, $\gamma < N-2$ if $N \ge 3$, $\gamma=0$ if $N=2$ and $$\mathcal{C}=\mathcal{C}(\gamma, \eta, h, ||\nabla u||_\infty, \rho, x_0, \alpha, M_A,c_2,\tau, \hat{C}).$$
	
	\noindent If $\Om$ is a smooth domain and $f$ is nonnegative in $\Om$
	\begin{equation}\label{reg2.1}
	\int_{\tOm} \frac{1}{(A(|\nabla u|))^{\alpha r}} \frac{1}{|x-y|^\gamma}dx \le \mathcal{C},
	\end{equation}
	for any $\tilde{\Om} \Subset \Om \setminus \{0\}$, $y\in \tOm$ and $\displaystyle \alpha:=\frac{p-1}{p-2}$, $r\in (0,1)$.
	
\end{thm}

\begin{proof}
	\noindent In this proof, in order to simplify the notation we will indicate by $\{A>1\}:=\{x\in B_{2\rho}(x_0) : A(|\nabla u(x)|)>1\}$ and by  $\{A<1\}:=\{x\in B_{2\rho}(x_0) : A(|\nabla u(x)|)<1\}$. Moreover, in the sequel
	$$
	 I_F:=\int_{B_{2\rho}(x_0)} \frac{1}{(A(|\nabla u|))^{\alpha r}} \frac{\varphi_\rho^2}{|x-y|^\gamma}dx
	$$
	Following \cite{CanDeGS, DamSciu,  EspST}, for $\varepsilon>0$, $\varphi_\rho$, $H_\delta$  defined as in Theorem \ref{Stima_D2_pt1}, let us define $\displaystyle \phi= \frac{H_\delta(|x-y|) \varphi_\rho^2}{(\varepsilon + A(|\nabla u|))^{\alpha r}} $.
Note that $\phi$ is a good test function, then
\begin{equation}
\label{term_th2}
\begin{split}
& \int_{B_{2\rho}(x_0)}  \frac{h(x,u) H_\delta \  \varphi_\rho^2 }{(\varepsilon + A(|\nabla u|))^{\alpha r}}\,dx\\
&= -\alpha r \int_{B_{2\rho}(x_0)} \frac{A(|\nabla u|)A'(|\nabla u|)H_\delta \ \varphi_\rho^2}{(\varepsilon + A(|\nabla u|))^{\alpha r+1}} (\nabla u, \tilde{\nabla} |\nabla u| ) \,dx\\
& \quad + \int_{B_{2\rho}(x_0)} \frac{A(|\nabla u|)\varphi_\rho^2 }{(\varepsilon + A(|\nabla u|))^{\alpha r}} (\nabla u, \nabla_x H_\delta)\,dx \\
&\quad + 2  \int_{B_{2\rho}(x_0)} \frac{A(|\nabla u|)H_\delta \ \varphi_\rho}{(\varepsilon + A(|\nabla u|))^{\alpha r}} (\nabla u, \nabla \varphi_\rho) \,dx+  \int_{B_{2\rho}(x_0)} \frac{B(|\nabla u|)H_\delta \  \varphi_\rho^2}{(\varepsilon + A(|\nabla u|))^{\alpha r}} \,dx.
\end{split}
\end{equation}
Since the source term $h$ is positive, calling $\displaystyle c_h(\rho):=\min_{B_{2\rho}}h(x,u)$ we deduce that
\begin{equation}
\label{estimate}
c_h(\rho)\int_{B_{2\rho}(x_0)}  \frac{H_\delta \  \varphi_{\rho}^2}{(\varepsilon + A(|\nabla u|))^{\alpha r}}  \,dx \le |I_1|+ |I_2|+ |I_3|+ |I_B|\,.
\end{equation}
where
\begin{equation}
\label{I_a_b_c_d}
\begin{split}
&I_1=   -\alpha r \int_{B_{2\rho}(x_0)} \frac{A(|\nabla u|)A'(|\nabla u|)}{(\varepsilon + A(|\nabla u|))^{(\alpha r+1)}} (\nabla u, \tilde{\nabla} |\nabla u| ) H_\delta \ \varphi_\rho^2 \,dx\\
&I_2=  \int_{B_{2\rho}(x_0)} \frac{A(|\nabla u|)}{(\varepsilon + A(|\nabla u|))^{\alpha r}} (\nabla u, \nabla_x H_\delta)\varphi_\rho^2 \,dx \\
&I_3= 2  \int_{B_{2\rho}(x_0)} \frac{A(|\nabla u|)}{(\varepsilon + A(|\nabla u|))^{\alpha r}} (\nabla u, \nabla \varphi_\rho)H_\delta \ \varphi_\rho \,dx \\
&I_B= \int_{B_{2\rho}(x_0)} \frac{B(|\nabla u|)}{(\varepsilon + A(|\nabla u|))^{\alpha r}} H_\delta \  \varphi_\rho^2\,dx.
\end{split}
\end{equation}
The idea of the proof is to show that all above integrals are bounded and then our estimate can be obtained by using Fatou Lemma.\\
By Young inequality $ab \le \tau a^2+ \frac{b^2}{4 \tau}$, we observe that
\begin{equation}
\label{IB}
\begin{split}
&\int_{B_{2\rho}(x_0)} \frac{B(|\nabla u|)}{(\varepsilon + A(|\nabla u|))^{\alpha r}} \frac{\varphi_\rho^2}{|x-y|^{\gamma}} \,dx \\\
&\le \int_{B_{2\rho}(x_0)} \frac{\varphi_{\rho}}{(\varepsilon + A(|\nabla u|))^\frac{{\alpha r}}{2}|x-y|^\frac{\gamma}{2}}   \ \frac{B(|\nabla u|)\varphi_{\rho}}{(\varepsilon + A(|\nabla u|))^\frac{{\alpha r}}{2}|x-y|^\frac{\gamma}{2}} \,dx \\
&\le \tau I_F+ \frac{1}{4 \tau} \int_{B_{2\rho}(x_0)} \frac{B^2(|\nabla u|)}{(A(|\nabla u|))^{\alpha r}} \ \frac{\varphi_{\rho}^2}{ |x-y|^{\gamma}} \,dx\\
& \le \tau I_F+ \frac{\hat{C}^2}{4 \tau} \int_{B_{2\rho}(x_0)} \frac{A^2(|\nabla u|)|\nabla u|^4}{(A(|\nabla u|))^{\alpha r}} \ \frac{ \varphi_{\rho}^2}{ |x-y|^{\gamma}}\,dx\\
& \le \tau I_F+ \frac{\hat{C}^2}{4 \tau}\left[ \int_{\{A<1\}}\frac{A^2(|\nabla u|)|\nabla u|^4}{(A(|\nabla u|))^{\alpha r}} \ \frac{\varphi_{\rho}^2}{ |x-y|^{\gamma}} \,dx
 +\int_{\{A>1\}}\frac{A^2(|\nabla u|)|\nabla u|^4}{(A(|\nabla u|))^{\alpha r}} \ \frac{\varphi_{\rho}^2}{ |x-y|^{\gamma}} \,dx\right].
\end{split}
\end{equation}
Since $p>2$ and $\alpha r>r>0$, we get that $A(|\nabla u|))<A(|\nabla u|))^{r}$ on the set $\{A<1\}$ therefore by \eqref{D20}
\begin{equation*}
\begin{split}
&\int_{B_{2\rho}(x_0)} \frac{B(|\nabla u|)}{(\varepsilon + A(|\nabla u|))^{\alpha r}} \frac{ \varphi_\rho^2}{|x-y|^{\gamma}} dx \\
 & \le \tau I_F+ \frac{\hat{C}^2}{4 \tau} \int_{\{A<1\}} \frac{A(|\nabla u|)|\nabla u|^4}{(A(|\nabla u|))^{\frac{r}{p-2}}} \ \frac{\varphi_{\rho}^2}{ |x-y|^{\gamma}} \,dx\\
& \quad +\frac{\hat{C}^2}{4 \tau} \int_{\{A>1\}} A^2(|\nabla u|)|\nabla u|^4\frac{\varphi_{\rho}^2}{ |x-y|^{\gamma}} \,dx \\
& \le \tau I_F+ \frac{\hat{C}^2}{4 \tau c_2^{\frac{r}{p-2}}} \int_{\{A<1\}} A(|\nabla u|)|\nabla u|^{4-r} \ \frac{\varphi_{\rho}^2}{ |x-y|^{\gamma}} \,dx\\
& \quad +\frac{\hat{C}^2}{4 \tau}\sup_{B_{2 \rho}(x_0)} (|\nabla u|^4 A^{2}(|\nabla u|)) \int_{\{A>1\}} \frac{\varphi_{\rho}^2}{ |x-y|^{\gamma}} \,dx\\
& \le \tau I_F+ \frac{\hat{C}^2}{4 \tau c_2^{\frac{r}{p-2}}} \sup_{B_{2 \rho}(x_0)} (|\nabla u|^{4-r} A(|\nabla u|)) \int_{\{A<1\}} \frac{\varphi_{\rho}^2}{ |x-y|^{\gamma}} \,dx\\
&\quad+\frac{\hat{C}^2}{4 \tau}\sup_{B_{2 \rho}(x_0)} (|\nabla u|^4 A^{2}(|\nabla u|)) \int_{\{A>1\}} \frac{\varphi_{\rho}^2}{ |x-y|^{\gamma}} \,dx.
\end{split}
\end{equation*}
To conclude, we can assert that there exists a positive $C_B$ such that
$$
\int_{B_{2\rho}(x_0)} \frac{B(|\nabla u|)}{(\varepsilon + A(|\nabla u|))^{\alpha r}} \frac{\varphi_\rho^2}{|x-y|^{\gamma}} \ dx \leq \tau I_F+\frac{C_B}{4\tau}
$$
hence by Fatou Lemma
\begin{equation}\label{IBp>2}
\begin{split}
\limsup_{\delta \to 0} |I_B| =\int_{B_{2\rho}(x_0)} \frac{B(|\nabla u|)}{(\varepsilon + A(|\nabla u|))^{\alpha r}} \frac{\varphi_\rho^2}{|x-y|^{\gamma}} \  dx &\leq \ \tau I_F+\frac{C_B}{4\tau}.
\end{split}
\end{equation}

\noindent Similarly, recalling that $|\nabla \varphi_\rho|\le \frac{2}{\rho}$, we get
\begin{equation}
\label{eqI3}
\begin{split}
& 2\int_{B_{2\rho}(x_0)}  \frac{A(|\nabla u|)|\nabla u| |\nabla \varphi|}{(\varepsilon + A(|\nabla u|))^{\alpha r}} \ \frac{\varphi_{\rho}}{ |x-y|^{\gamma}} \,dx \leq \frac{4}{\rho}\int_{B_{2\rho}(x_0)} \frac{A(|\nabla u|)|\nabla u|}{(\varepsilon + A(|\nabla u|))^{\alpha r}} \ \frac{\varphi_{\rho}}{ |x-y|^{\gamma}} \,dx. \\
\end{split}
\end{equation}

Again we divide the last integral on the set $\{A>1\}$ and $\{A<1\}$ respectively and we get that
$$
\int_{\{A>1\}}\frac{A(|\nabla u|)|\nabla u|}{(\varepsilon + A(|\nabla u|))^{\alpha r}} \frac{\varphi_{\rho}}{ |x-y|^{\gamma}} \,dx\leq \sup_{B_{2 \rho}(x_0)} (|\nabla u| A(|\nabla u|)) \int_{B_{2\rho}(x_0)} \frac{\varphi_{\rho}}{ |x-y|^{\gamma}} \,dx
$$
while, following arguing as for $I_B$, being $\alpha r>r$ and $A(|\nabla u|))<A(|\nabla u|))^{r}$
$$
\int_{\{A<1\}}\frac{A(|\nabla u|)|\nabla u|}{(\varepsilon + A(|\nabla u|))^{\alpha r}}\frac{\varphi_{\rho}}{ |x-y|^{\gamma}} \,dx\leq \sup_{B_{2 \rho}(x_0)} (|\nabla u|^{1-r}) \int_{B_{2\rho}(x_0)} \frac{\varphi_{\rho}}{ |x-y|^{\gamma}} \,dx.
$$
Therefore we can state that there exists $C_3>0$ such that
$$
2\int_{B_{2\rho}(x_0)}  \frac{A(|\nabla u|)|\nabla u| |\nabla \varphi|}{(\varepsilon + A(|\nabla u|))^{\alpha r}} \ \frac{1}{ |x-y|^{\gamma}} \varphi_{\rho}\,dx\leq C_3
$$
and then
$$
\limsup_{\delta \to 0} |I_3|\leq C_3.
$$
\noindent Now we proceed further observing that, with the same procedure, we can show that there exists $C_2>0$ such that
$$
\limsup_{\delta \to 0} |I_2|\leq C_2.
$$
It remain to consider $I_1$: again Young inequality give us that
\begin{equation*}
\begin{split}
&\limsup_{\delta \to 0}|I_1| \le  \alpha r \int_{B_{2\rho}(x_0)} \frac{A(|\nabla u|) |A'(|\nabla u|)| |\nabla u|}{(\varepsilon + A(|\nabla u|))^{\alpha r+1}}  \frac{\norm{D^2u}}{|x-y|^\gamma}\ \varphi_\rho^2\,dx \\
&= \alpha r\int_{B_{2\rho}(x_0)} \frac{\varphi_\rho }{(\varepsilon + A(|\nabla u|))^\frac{\alpha r}{2}|x-y|^\frac{\gamma}{2}}  \ \frac{A(|\nabla u|) |A'(|\nabla u|)| |\nabla u|}{(\varepsilon + A(|\nabla u|))^{\frac{\alpha r}{2}+1}} \frac{\norm{D^2u}}{|x-y|^\frac{\gamma}{2}}\ \varphi_\rho \,dx \\
&\le \alpha \tau  I_F +\frac{\alpha r}{4 \tau} \int_{B_{2\rho}(x_0)} \frac{(A(|\nabla u|))^2 (A'(|\nabla u|))^2 |\nabla u|^2}{(\varepsilon + A(|\nabla u|))^{\alpha r+2}} \frac{\norm{D^2u}^2}{|x-y|^\gamma}\ \varphi_\rho^2 \,dx \\
&\leq \alpha r\tau  I_F + \frac{\alpha r M_A^2}{4 \tau}\int_{B_{2\rho}(x_0)}\frac{(A(|\nabla u|))^2 }{(\varepsilon + A(|\nabla u|))^{\alpha r}}  \frac{\norm{D^2u}^2}{|x-y|^\gamma}\ \varphi_\rho^2 \,dx \\
&\leq \alpha r \tau  I_F + \frac{\alpha r M_A^2}{4 \tau}\int_{B_{2\rho}(x_0)}\frac{(A(|\nabla u|))^2 }{ A(|\nabla u|))^{\alpha r}}  \frac{\norm{D^2u}^2}{|x-y|^\gamma}\ \varphi_\rho^2 \,dx. \\
\end{split}
\end{equation*}
\vskip 10pt
\noindent As in previous cases, in the set  $\{A<1\}$, we easily get
\begin{equation*}
\begin{split}
 &\int_{\{A<1\}}\frac{(A(|\nabla u|))^2 }{ A(|\nabla u|))^{\alpha r}}  \frac{\norm{D^2u}^2}{|x-y|^\gamma}\ \varphi_\rho^2 \,dx\leq \int_{\{A<1\}}\frac{A(|\nabla u|) }{ A(|\nabla u|))^{\frac{r}{p-2}}}  \frac{\norm{D^2u}^2}{|x-y|^\gamma}\ \varphi_\rho^2 \,dx\\
& \quad \leq \frac{1}{c_2^{\frac{r}{p-2}}} \int_{ \{A<1\}}\frac{A(|\nabla u|)}{ |\nabla u|^{r}}  \frac{\norm{D^2u}^2}{|x-y|^\gamma}\ \varphi_\rho^2 \,dx\\
\end{split}
\end{equation*}
so we can use Theorem \ref{Stima_D2_pt1} with $\beta=0$.
On the set $\{A>1\}$, 
\begin{equation*}
\begin{split}
 &\int_{\{A>1\}}\frac{(A(|\nabla u|))^2 }{ A(|\nabla u|))^{\alpha r}}  \frac{\norm{D^2u}^2}{|x-y|^\gamma}\ \varphi_\rho^2 \,dx\leq \int_{\{A>1\}}A(|\nabla u|)^{1-r}  A(|\nabla u|)\frac{\norm{D^2u}^2}{|x-y|^\gamma}\ \varphi_\rho^2 \,dx\\
& \quad \leq \sup_{B_{2\rho}}(A(|\nabla u| |\nabla u|^{\eta})^{1-r} \int_{\{A>1\}}\frac{A(|\nabla u|)}{ |\nabla u|^{\eta(1-r)}}  \frac{\norm{D^2u}^2}{|x-y|^\gamma}\ \varphi_\rho^2 \,dx\\
\end{split}
\end{equation*}
hence we can  use again Theorem \ref{Stima_D2_pt1} with $\beta=\eta(1-r)$.
Summarizing we obtain the existence of a constant $C_1>0$ such that
\begin{equation}\label{I1p>2}	
   \limsup_{\delta \to 0}|I_1| \le \alpha r \tau  I_F + \frac{\alpha r M_A^2}{4 \tau} C_1.
\end{equation} 

\vskip 10pt
\noindent Collecting all our estimates, from \eqref{estimate} and letting $\delta \to 0$  we have for every $r \in (0,1)$
\begin{equation*}
\begin{split}
&(c_h(\rho)- (\alpha r+1) \tau) \int_{B_{2\rho}(x_0)} \frac{1}{(\varepsilon + A(|\nabla u|))^{\alpha r} |x-y|^\gamma }  \varphi_{\rho}^2\,dx 
\le  \frac{\alpha r M_A^2 C_1}{4\tau}+C_2+C_3+\frac{C_B}{4\tau}.
\end{split}
\end{equation*}
For $\tau$ sufficient small such that $(c_h(\rho)- (\alpha r+1) \tau)>0$, letting $\varepsilon \to 0$, we get by Fatou Lemma that

\begin{equation}
\begin{split}
\int_{B_{2\rho}(x_0)}  \frac{ \varphi_{\rho}^2}{A(|\nabla u|)^{\alpha r} |x-y|^\gamma } \,dx \leq \lim\limits_{\varepsilon \rightarrow 0} \int_{B_{2\rho}(x_0)} \frac{\varphi_{\rho}^2}{(\varepsilon + A(|\nabla u|))^{\alpha r} |x-y|^\gamma }  \,dx  \leq \mathcal{C}
\end{split}
\end{equation}

\noindent where $\mathcal{C}=\mathcal{C}(\gamma, \beta, h, ||\nabla u||_\infty, \rho, x_0, \alpha, M_A,c_2,\tau, \hat{C})$.

\noindent Observing that  $|Z_u|=0$,
%
%
%
%
we obtain 
\begin{equation*}
 \int_{B_{\rho}(x_0)} \frac{1}{(A(|\nabla u|))^{\alpha r}}\frac{1}{ |x-y|^\gamma }\,dx  \le  \int_{B_{2\rho}(x_0)} \frac{1}{(A(|\nabla u|))^{\alpha r} |x-y|^\gamma }  \varphi_{\rho}^2\,dx \le \mathcal{C}.
\end{equation*}

\vskip 10pt
\noindent Then, we obtain that 

\begin{equation}
\int_{B_{\rho}(x_0)} \frac{1}{(A(|\nabla u|))^{\alpha r}}\frac{1}{ |x-y|^\gamma }\,dx \le \mathcal{C}.
\end{equation}

If $\Om$ is a smooth domain and $f$ is nonnegative in $\Om$, by H\"{o}pf Lemma we get that $Z_u \cap \partial \Om = \emptyset$, then in a neighbourhood of $\partial \Om$, $|\nabla u|$ is strictly positive. Arguing as in \cite[Theorem 4.2]{CanDeGS} we get \eqref{reg2.1}.

\end{proof}

Similarly to Theorem \ref{Local_sum_weight}, we now show summability properties for $A(|\nabla u|)^{-1}$ in the case $1<p< 2$.

\begin{rem}
	If $1<p < 2$, we have that $m_A<0$. In fact, by \eqref{D20}, one has that 
	$$ A(|\nabla u|) \geq \frac{c_2}{|\nabla u|^{2-p}}\xrightarrow{|\nabla u| \rightarrow 0} +\infty.$$
	Then if $|\nabla u| \rightarrow 0$, considering also \cite[Proposition 4.1]{Cia14}, we get
	$$\displaystyle  \frac{c_2}{|\nabla u|^{2-p}}\leq A(|\nabla u|) \leq A(1)|\nabla u|^{m_A}$$ and so $m_A$ has to be necessarily negative.
	Therefore, from \eqref{hpregolarita}, we get that $\alpha:= \frac{m_A+1}{m_A}<0$. 
\end{rem}
\begin{thm}
	\label{Local_sum_weight2}
	Let $u\in C^{1,\alpha}(\Omega\setminus \{0\})\cap C^{2}(\Omega\setminus (Z_u \cup \{0\}))$ be a solution to \eqref{ProbPuntuale} with $f\in W^{1,\infty}(\Omega)$ and
	$f(x)>0$ in $B_{2 \rho}(x_0) \subset \Omega$. Let us suppose that $p\in(1,2)$. Then 
	\begin{equation}
	\label{reg_2}
	\int_{B_{\rho}(x_0)} \frac{1}{(A(|\nabla u|))^{\alpha r}} \frac{1}{|x-y|^\gamma} \le \mathcal{C}
	\end{equation}
	for any $y\in B_{\rho}(x_0)$, with $\displaystyle \alpha:=\frac{m_A+1}{m_A}$, $r\in (0,1)$, $\gamma < N-2$ if $N \ge 3$, $\gamma=0$ if $N=2$ and $$\mathcal{C}=\mathcal{C}(\gamma, \beta, h, ||\nabla u||_\infty, \rho, x_0, \alpha, M_A,c_2,\tau, \hat{C}).$$
	
	\noindent If $\Om$ is a smooth domain and $f$ is nonnegative in $\Om$
	\begin{equation}\label{reg2.1p<2}
	\int_{\tOm} \frac{1}{(A(|\nabla u|))^{\alpha r}} \frac{1}{|x-y|^\gamma} \le \mathcal{C}
	\end{equation}
	for any $\tOm \Subset \Om \setminus \{0\}$, $y\in \tOm$ and $\displaystyle \alpha:=\frac{m_A+1}{m_A}$, $r\in (0,1)$.
	
\end{thm}

\begin{proof}
	\noindent Choose the same test function as in Theorem \ref{Local_sum_weight} (taking into account the new value of $\alpha$) and let us start from the inequality
%
%
	\begin{equation}
	\label{estimate2}
	c_h(\rho)\int_{B_{2\rho}(x_0)}  \frac{H_\delta \  \varphi_{\rho}^2 }{(\varepsilon + A(|\nabla u|))^{\alpha r}} \,dx \le |I_1|+ |I_2|+ |I_3|+ |I_B|\,.
	\end{equation}
	by using the same notations as in proof of Theorem \ref{Local_sum_weight}.\ \\
	\noindent Let $\bar{\varepsilon}>0$. To our goal, we will proceed in two different ways depending on whether we work on the set $Z_{u}^{\bar{\varepsilon}}:=\{x\in B_{2\rho}(x_0): d(x,Z_u)<\bar{\varepsilon}\}$ or $B_{2\rho}(x_0) \setminus \Zueps$.\\ If $x\in \Zueps$, by \eqref{D20}, there exists $K\geq 1$ such that $A(|\nabla u(x)|) > K$. Then, we will use later the following estimate
	\begin{equation}\label{Zue}
	(\varepsilon + A(|\nabla u|))^{-\alpha r} < (K+A(|\nabla u|))^{-\alpha r} < (K+A(|\nabla u|))^{-\alpha} <  2^{-\alpha}A(|\nabla u|)^{-\alpha};
	\end{equation}
	
	\noindent while, if $x\in B_{2\rho}(x_0) \setminus \Zueps$ we have that $\sup\limits_{B_{2\rho}(x_0) \setminus \Zueps} A(|\nabla u(x)|) < +\infty.$ Then, there exists $C\in \mathbb{R}^+$ such that
	\begin{equation}\label{Zue^C}
	(\varepsilon+A(|\nabla u|))^{-\alpha r}< 2^{-(\alpha r +1)}(\varepsilon^{-\alpha r} + A(|\nabla u|)^{-\alpha r})< C.
	\end{equation}
	By Young inequality $ab \le \tau a^2+ \frac{b^2}{4 \tau}$, we observe that
	\begin{equation}
	\label{IB2}
	\begin{split}
	&\int_{B_{2\rho}(x_0)} \frac{B(|\nabla u|)}{(\varepsilon + A(|\nabla u|))^{\alpha r}} \frac{\varphi_\rho^2}{|x-y|^{\gamma}} \ \,dx \\\
	&\le \int_{B_{2\rho}(x_0)} \frac{\varphi_{\rho}}{(\varepsilon + A(|\nabla u|))^\frac{{\alpha r}}{2}|x-y|^\frac{\gamma}{2}}   \ \frac{\varphi_{\rho}B(|\nabla u|)}{(\varepsilon + A(|\nabla u|))^\frac{{\alpha r}}{2}|x-y|^\frac{\gamma}{2}} \,dx \\
	&\le \tau I_F+ \frac{1}{4 \tau} \int_{B_{2\rho}(x_0)} \frac{B^2(|\nabla u|)}{(\varepsilon+A(|\nabla u|))^{\alpha r}} \ \frac{\varphi_{\rho}^2}{ |x-y|^{\gamma}} \,dx\\
	& \le \tau I_F+ \frac{\hat{C}^2}{4 \tau} \int_{B_{2\rho}(x_0)} \frac{A^2(|\nabla u|)|\nabla u|^4}{(\varepsilon+A(|\nabla u|))^{\alpha r}} \ \frac{\varphi_{\rho}^2}{ |x-y|^{\gamma}}\,dx\\
	& \le \tau I_F+ \frac{\hat{C}^2}{4 \tau} \int_{B_{2\rho}(x_0) \setminus \Zueps} \frac{A^2(|\nabla u|)|\nabla u|^4}{(\varepsilon+A(|\nabla u|))^{\alpha r}} \ \frac{\varphi_{\rho}^2}{ |x-y|^{\gamma}} dx\\ 
	& \quad +\frac{\hat{C}^2}{4 \tau} \int_{\Zueps} \frac{A^2(|\nabla u|)|\nabla u|^4}{(\varepsilon+A(|\nabla u|))^{\alpha r}} \ \frac{\varphi_{\rho}^2}{ |x-y|^{\gamma}} \,dx\\
	& \le \tau I_F+ \frac{\hat{C}^2}{4 \tau}C \sup_{B_{2 \rho}(x_0) \setminus \Zueps} (|\nabla u|^4 A^{2}(|\nabla u|))\int_{B_{2 \rho}(x_0) \setminus \Zueps}  \frac{\varphi_{\rho}^2}{ |x-y|^{\gamma}} \,dx\\
	& \quad +\frac{\hat{C}^2}{4 \tau 2^{\alpha}} \int_{\Zueps} A^2(|\nabla u|)|\nabla u|^4 A(|\nabla u|)^{-\alpha} \frac{\varphi_{\rho}^2}{ |x-y|^{\gamma}} \,dx.
	\end{split}
	\end{equation}
	
	\noindent By \cite[Proposition 4.1]{Cia14} we have that $A(|\nabla u|) \leq A(1)|\nabla u|^{m_A}$ on $\Zueps$ then
	\begin{equation}
	\begin{split}
	&\int_{B_{2\rho}(x_0)} \frac{B(|\nabla u|)}{(\varepsilon + A(|\nabla u|))^{\alpha r}} \frac{\varphi_\rho^2}{|x-y|^{\gamma}} \  \,dx \\
	&\le \tau I_F+ \frac{\hat{C}^2}{4 \tau}C \sup_{B_{2 \rho}(x_0) \setminus \Zueps} (|\nabla u|^4 A^{2}(|\nabla u|))\int_{B_{2 \rho}(x_0) \setminus \Zueps}  \frac{\varphi_{\rho}^2}{ |x-y|^{\gamma}} \,dx\\
	& \quad +\frac{\hat{C}^2A(1)^{2-\alpha}}{4 \tau 2^{\alpha}} \int_{\Zueps} \frac{ |\nabla u|^{m_A+3}\varphi_{\rho}^2}{ |x-y|^{\gamma}} \,dx\\
	&\le \tau I_F+ \frac{\hat{C}^2}{4 \tau}C \sup_{B_{2 \rho}(x_0) \setminus \Zueps} (|\nabla u|^4 A^{2}(|\nabla u|))\int_{B_{2 \rho}(x_0) \setminus \Zueps}  \frac{\varphi_{\rho}^2}{ |x-y|^{\gamma}} \,dx\\
	& \quad +\frac{\hat{C}^2A(1)^{2-\alpha}}{4 \tau 2^{\alpha}}\sup_{\Zueps}|\nabla u|^{m_A+3} \int_{\Zueps} \frac{\varphi_{\rho}^2}{ |x-y|^{\gamma}} \,dx.\\
	\end{split}
	\end{equation}
	
	\noindent To conclude, we can assert that there exists a positive $C_B$ such that
	$$
	\int_{B_{2\rho}(x_0)} \frac{B(|\nabla u|)}{(\varepsilon + A(|\nabla u|))^{\alpha r}} \frac{1}{|x-y|^{\gamma}} \  \varphi_\rho^2\,dx \leq \tau I_F+\frac{C_B}{4\tau}
	$$
	hence by Fatou Lemma
	\begin{equation}\label{IB2p<2}
	\begin{split}
	\limsup_{\delta \to 0} |I_B| &\leq \ \tau I_F+\frac{C_B}{4\tau}.
	\end{split}
	\end{equation}
	
	\noindent Similarly, recalling that $|\nabla \varphi_\rho|\le \frac{2}{\rho}$, we get
	\begin{equation}
	\label{eqI32}
	\begin{split}
	& 2\int_{B_{2\rho}(x_0)}  \frac{A(|\nabla u|)|\nabla u| |\nabla \varphi|}{(\varepsilon + A(|\nabla u|))^{\alpha r}} \ \frac{\varphi_{\rho}}{ |x-y|^{\gamma}} \,dx \leq \frac{4}{\rho}\int_{B_{2\rho}(x_0)} \frac{A(|\nabla u|)|\nabla u|}{(\varepsilon + A(|\nabla u|))^{\alpha r}} \ \frac{\varphi_{\rho}}{ |x-y|^{\gamma}} \,dx. \\
	\end{split}
	\end{equation}
	
	\noindent Again we divide the last integral on the set $\Zueps$ and $B_{2\rho}(x_0) \setminus \Zueps$ respectively and we get, using \eqref{Zue}, that

	\begin{equation*}
	\begin{split}
	&\int_{\Zueps}\frac{A(|\nabla u|)|\nabla u|}{(\varepsilon + A(|\nabla u|))^{\alpha r}} \frac{\varphi_{\rho}}{ |x-y|^{\gamma}} \,dx\leq 2^{-\alpha} \int_{\Zueps} \frac{A(|\nabla u|)^{1-\alpha}|\nabla u|\varphi_{\rho}}{ |x-y|^{\gamma}} \,dx\\
	&\le 2^{-\alpha}A(1)^{1-\alpha} \int_{\Zueps} \frac{|\nabla u|^{m_A(1-\alpha)}|\nabla u|\varphi_{\rho}}{ |x-y|^{\gamma}} \,dx = 2^{-\alpha}A(1)^{1-\alpha} \int_{\Zueps} \frac{\varphi_{\rho}}{ |x-y|^{\gamma}} \,dx;
	\end{split}
    \end{equation*}

	\noindent while, using \eqref{Zue^C} we get 
	\begin{equation*}
	\begin{split}
	&\int_{B_{2\rho}(x_0)\setminus \Zueps }\frac{A(|\nabla u|)|\nabla u|}{(\varepsilon + A(|\nabla u|))^{\alpha r}}\frac{\varphi_{\rho}}{ |x-y|^{\gamma}} \,dx \\ 
	& \leq C \sup_{B_{2 \rho}(x_0) \setminus \Zueps} (A(|\nabla u|)|\nabla u|) \int_{B_{2\rho}(x_0)\setminus \Zueps} \frac{\varphi_{\rho}}{ |x-y|^{\gamma}} \,dx.
	\end{split}
	\end{equation*}

	Therefore we can state that there exists $C_3>0$ such that
	$$
	2\int_{B_{2\rho}(x_0)}  \frac{A(|\nabla u|)|\nabla u| |\nabla \varphi|}{(\varepsilon + A(|\nabla u|))^{\alpha r}} \ \frac{1}{ |x-y|^{\gamma}} \varphi_{\rho}\,dx\leq C_3
	$$
	and then
	$$
	\limsup_{\delta \to 0} |I_3|\leq C_3.
	$$
	\noindent Now we proceed further observing that, with the same procedure, we can show that there exists $C_2>0$ such that
	$$
	\limsup_{\delta \to 0} |I_2|\leq C_2.
	$$
	It remain to consider $I_1$: again Young inequality give us that
	\begin{equation*}
	\begin{split}
	&\limsup_{\delta \to 0}|I_1| \le  \alpha r \int_{B_{2\rho}(x_0)} \frac{A(|\nabla u|) |A'(|\nabla u|)| |\nabla u|}{(\varepsilon + A(|\nabla u|))^{\alpha r+1}}  \frac{\norm{D^2u}}{|x-y|^\gamma}\ \varphi_\rho^2\,dx \\
	&= \alpha r\int_{B_{2\rho}(x_0)} \frac{1}{(\varepsilon + A(|\nabla u|))^\frac{\alpha r}{2}|x-y|^\frac{\gamma}{2}}  \varphi_\rho \ \frac{A(|\nabla u|) |A'(|\nabla u|)| |\nabla u|}{(\varepsilon + A(|\nabla u|))^{\frac{\alpha r}{2}+1}} \frac{\norm{D^2u}}{|x-y|^\frac{\gamma}{2}}\ \varphi_\rho \,dx \\
	&\le \alpha \tau  I_F +\frac{\alpha r}{4 \tau} \int_{B_{2\rho}(x_0)} \frac{(A(|\nabla u|))^2 (A'(|\nabla u|))^2 |\nabla u|^2}{(\varepsilon + A(|\nabla u|))^{\alpha r+2}} \frac{\norm{D^2u}^2}{|x-y|^\gamma}\ \varphi_\rho^2 \,dx \\
	&\leq \alpha r\tau  I_F + \frac{\alpha r M_A^2}{4 \tau}\int_{B_{2\rho}(x_0)}\frac{(A(|\nabla u|))^2 }{(\varepsilon + A(|\nabla u|))^{\alpha r}}  \frac{\norm{D^2u}^2}{|x-y|^\gamma}\ \varphi_\rho^2 \,dx 
	\end{split}
	\end{equation*}
	\vskip 10pt
	\noindent As in previous cases, using \eqref{Zue^C} and choosing $\eta$ as in Remark \ref{teAlim}, we easily get
	\begin{equation*}
	\begin{split}
	&\int_{B_{2\rho}(x_0)\setminus \Zueps}\frac{(A(|\nabla u|))^2 }{( \varepsilon+A(|\nabla u|))^{\alpha r}}  \frac{\norm{D^2u}^2}{|x-y|^\gamma}\ \varphi_\rho^2 \,dx\\
	& \leq C \int_{B_{2\rho}(x_0)\setminus \Zueps}\frac{A(|\nabla u|)\norm{D^2u}^2 }{ |\nabla u|^{\eta}|x-y|^{\gamma}} A(|\nabla u|)|\nabla u|^{\eta}\ \varphi_\rho^2 \,dx
	\end{split}
	\end{equation*}
	so we can use Theorem \ref{Stima_D2_pt1} and Remark \ref{teAlim} to conclude that this term is finite.
	On the set $\Zueps$, by \eqref{Zue}, one has 
	\begin{equation*}
	\begin{split}
	&\int_{\Zueps}\frac{(A(|\nabla u|))^2 }{ (\varepsilon + A(|\nabla u|))^{\alpha r}}  \frac{\norm{D^2u}^2}{|x-y|^\gamma}\ \varphi_\rho^2 \,dx\leq 2^{-\alpha r}\int_{\Zueps}\frac{A(|\nabla u|)\norm{D^2u}^2 }{ |x-y|^{\gamma}} A(|\nabla u|)^{1-\alpha r}\ \varphi_\rho^2 \,dx\\
	&\leq 2^{-\alpha r}A(1)^{1-\alpha r} \int_{\Zueps}\frac{A(|\nabla u|) \norm{D^2u}^2}{|x-y|^\gamma} |\nabla u|^{m_A(1-\alpha r)}\ \varphi_\rho^2 \,dx\\
	& =  2^{-\alpha r}A(1)^{1-\alpha r} \int_{\Zueps}\frac{A(|\nabla u|) \norm{D^2u}^2}{|x-y|^\gamma |\nabla u|^{m_A(\alpha r-1)}} \ \varphi_\rho^2 \,dx
	\end{split}
	\end{equation*}
	hence we can  use again Theorem \ref{Stima_D2_pt1} with $\beta=m_A(\alpha r -1)$.
	Summarizing we obtain the existence of a constant $C_1>0$ such that
	\begin{equation}\label{I12p<2}	
	\limsup_{\delta \to 0}|I_1| \le \alpha r \tau  I_F + \frac{\alpha r M_A^2}{4 \tau} C_1.
	\end{equation} 
	
	\vskip 10pt
	\noindent Collecting all our estimates, from \eqref{estimate2} and letting $\delta \to 0$  we have for every $r \in (0,1)$
	\begin{equation*}
	\begin{split}
	&(c_h(\rho)- (\alpha r+1) \tau) \int_{B_{2\rho}(x_0)} \frac{1}{(\varepsilon + A(|\nabla u|))^{\alpha r} |x-y|^\gamma }  \varphi_{\rho}^2\,dx 
	\le  \tilde{C}_1+C_2+C_3+\tilde{C}_B.
	\end{split}
	\end{equation*}
	For $\tau$ sufficient small such that $(c_h(\rho)- (\alpha r+1) \tau)>0$, letting $\varepsilon \to 0$, we get by Fatou Lemma that

	\begin{equation}
	\begin{split}
	&\int_{B_{2\rho}(x_0)} \frac{\varphi_{\rho}^2}{A(|\nabla u|)^{\alpha r} |x-y|^\gamma }  \,dx \leq \lim\limits_{\varepsilon \rightarrow 0} \int_{B_{2\rho}(x_0)} \frac{\varphi_{\rho}^2\,dx}{(\varepsilon + A(|\nabla u|))^{\alpha r} |x-y|^\gamma }    \leq \mathcal{C}
	\end{split}
	\end{equation}
	
	\noindent where $\mathcal{C}=\mathcal{C}(\gamma, \beta, h, ||\nabla u||_\infty, \rho, x_0, \alpha, M_A,c_2,\tau, \hat{C})$.
	
%
%
%
	\noindent Therefore, we obtain 
	\begin{equation*}
	\int_{B_{\rho}(x_0)} \frac{1}{(A(|\nabla u|))^{\alpha r}}\frac{1}{ |x-y|^\gamma }\,dx  \le  \int_{B_{2\rho}(x_0)} \frac{1}{(A(|\nabla u|))^{\alpha r} |x-y|^\gamma }  \varphi_{\rho}^2\,dx \le \mathcal{C}.
	\end{equation*}
	i.e. the claimed.
%
	
	If $\Om$ is a smooth domain and $f$ is nonnegative in $\Om$, by H\"{o}pf Lemma we get that $Z_u \cap \partial \Om = \emptyset$, then in a neighbourhood of $\partial \Om$, $|\nabla u|$ is strictly positive. Arguing as in \cite[Theorem 4.2]{CanDeGS} we get \eqref{reg2.1p<2}.
\end{proof}

Theorem \ref{Local_sum_weight} and \ref{Local_sum_weight2} state that in case of $A(t)=t^{p-2}$ for $p>1$, we exactly find \cite[Theorem 2.3]{DamSciu}; the same choice made in \cite[Theorem 4.2]{CanDeGS} for $p>2$ give us a closer interval for the power $\alpha$  and a different interval for $p\in (1,2)$.\\

\begin{prop}
	Let $\Om$ be a smooth domain, $\tOm \Subset \Om \setminus \{0\}$, $u \in C^1(\Om \setminus \{0\})$ be a weak solution to \eqref{ProbPuntuale} and suppose that $f$ is nonnegative and $f\in W^{1,\infty}(\bar{\Om})$. Then $|\displaystyle u_{ij}|\in L^2(\tOm)$ if $1<p<3$. If $p\geq 3$, then $|\displaystyle u_{ij}| \in L^q(\tOm)$ with  $\displaystyle q< \frac{p-1}{p-2}$.
\end{prop}

\begin{proof}
	By Remark \ref{IstimasuOm}, for every $\beta<1$ we have
	$$|\nabla u|^{\frac{p-2-\beta}{2}}u_{ij} \in L^2(\tOm). $$
	
	\noindent Then, if $1<p<3$ we can choose $\beta$ such that $p-2-\beta<0$ obtaining that $\displaystyle u_{ij}\in L^2(\tOm)$.
	
	\noindent Instead, for the case $p\geq 3$, 

by Holder inequality with conjugate exponents $\frac{2}{q}$ and $\frac{2}{2-q}$, for $\eta \in [0, 1)$ we can apply Remark \ref{IstimasuOm}, obtaining
	
	\begin{equation*}
	\begin{split}
	\int_{\tOm} |u_{ij}|^q\,dx &= \int_{\tOm} \frac{|u_{ij}|^q A(|\nabla u|)^{\frac{q}{2}}}{|\nabla u|^{\frac{\eta q}{2}}}\frac{|\nabla u|^{\frac{\eta q}{2}}}{A(|\nabla u|)^{\frac{q}{2}}} \,dx \\
	&\leq \left(\int_{\tOm} \frac{|u_{ij}|^2 A(|\nabla u|)}{|\nabla u|^{\eta}}\,dx \right)^{\frac{q}{2}} \left(\bigintssss_{\tOm} \frac{|\nabla u|^{\frac{\eta q}{2-q}}}{A(|\nabla u|)^{\frac{q}{2-q}}} \,dx\right)^{\frac{2-q}{2}} \\
	& \leq c\mathcal{C}^{\frac{q}{2}} \left(\bigintssss_{\tOm} \frac{A(|\nabla u|)^{\frac{\eta q}{(2-q)(p-2)}}}{A(|\nabla u|)^{\frac{q}{2-q}}} \,dx\right)^{\frac{2-q}{2}} = c\mathcal{C}^{\frac{q}{2}} \left(\bigintssss_{\tOm} \frac{1}{A(|\nabla u|)^{\frac{q}{2-q}\left(1-\frac{\eta}{p-2}\right)}} \,dx\right)^{\frac{2-q}{2}}.
	\end{split}
	\end{equation*}

\noindent For the last integral, we show that if $q< \frac{p-1}{p-2}$, we can choose $\eta \in [0,1)$ such that 
\begin{equation}\label{exp}
\frac{q}{2-q}\left(1-\frac{\eta}{p-2}\right) < \frac{p-1}{p-2};
\end{equation}
hence, exploiting Theorem \ref{Local_sum_weight}, we get
\begin{equation*}
\left(\bigintssss_{\tOm} \frac{1}{A(|\nabla u|)^{\frac{q}{2-q}\left(1-\frac{\eta}{p-2}\right)}} \,dx\right)^{\frac{2-q}{2}} <+\infty.
\end{equation*}

\noindent In fact, let be consider $\varepsilon >0$ and $q= \frac{p-1}{p-2+\varepsilon} < \frac{p-1}{p-2}$ in \eqref{exp}. Observing that $q<\frac{p-1}{p-2} <2$, the left-hand side of \eqref{exp} becomes

\begin{equation*}
\frac{q}{2-q}\left(1-\frac{\eta}{p-2}\right) = \frac{p-1}{p-2} \left(\frac{p-2-\eta}{p-3+2\varepsilon}\right).
\end{equation*}

\noindent Choosing $\eta> 1-2\varepsilon$, we get that 
$$\left(\frac{p-2-\eta}{p-3+2\varepsilon}\right) <1; $$

\noindent hence, for this choice of $\eta$, \eqref{exp} is verified.
Therefore, $|\displaystyle u_{ij}| \in L^q(\tOm)$ with  $\displaystyle q< \frac{p-1}{p-2}$.

\end{proof}
 
\noindent Arguing as in \cite[Proposition 2.2]{DamSciu}, the following can be proved: 
\begin{thm}\label{PropforStamp}
	Let $\Om$ be a smooth domain, $u\in C^1(\bar{\Om} \setminus \{0\})$ be a weak solution to \eqref{ProbPuntuale}, $f \in W^{1,\infty}(\bar{\Om})$ and $f$ nonnegative function. Then if $1<p<3$, $u_{i} \in W^{1,2}_{loc}(\Om \setminus \{0\})$, while if $p\geq 3$ then $u_{i} \in W^{1,q}_{loc}(\Om \setminus \{0\})$, for every $i=1,...,N$ and for every $\displaystyle q< \frac{p-1}{p-2}$. Moreover the generalized derivatives of $u_{i}$ coincide with the classical ones, both denoted with $u_{ij}$ almost everywhere in $\Om$.
\end{thm}

\begin{rem}
	Since $q< \frac{p-1}{p-2} < 2$, by Theorem \ref{PropforStamp}, we have that for every $i=1,...,N$
	$$ u_{i} \in W^{1,q}_{loc}(\Om \setminus \{0\}) \quad \mbox{ for } 1<p<\infty.$$
	
	\noindent By Stampacchia Theorem \cite[Theorem 1.56]{Troianello},  we get that $$\nabla u_{i}=0 \mbox{ a.e. in } \{u_{i}=0\}, \mbox{ for } i=1,...,N.$$
	
	\noindent Then, the choice to set zero the second derivatives of $u$ in \eqref{DerSeconde} is completely justified.
\end{rem}

\section{Declarations}

{\bf Competing interests.}
The authors declare that they have no competing interests.\ \\

{\bf Acknowledgment.}

{\bf Authors' contributions.}
Each author equally contributed to this paper, read and approved the final manuscript.

\end{document}